\documentclass[11pt, twoside]{article}
\usepackage{amsfonts,amsmath,amssymb,amsthm}
\usepackage{mathrsfs}
\usepackage{amssymb,amscd,amsfonts,amsbsy}
\usepackage{color}
\usepackage{latexsym}
\usepackage{txfonts}
\usepackage{indentfirst}
\usepackage{soul}
\usepackage{enumitem}
\usepackage{anysize}

\usepackage{hyperref}
\hypersetup{colorlinks=true,linkcolor=blue, anchorcolor=green, citecolor=blue, urlcolor=red, filecolor=magenta, pdftoolbar=true}

\allowdisplaybreaks

\def\BMO{\operatorname{BMO}}

\def\max{\operatorname{max}}

\def\supp{\operatorname{supp}}

\def\rn{\mathbb{R}^n}
\def\la{\lambda}
\def\va{\varphi}

\newtheorem{theorem}{Theorem}

\newtheorem{proposition}[theorem]{Proposition}

\newtheorem{corollary}[theorem]{Corollary}
\newtheorem{lemma}[theorem]{Lemma}

\newtheorem{remark}[theorem]{Remark}
\newtheorem{definition}[theorem]{Definition}

\theoremstyle{definition}

\numberwithin{theorem}{section} \numberwithin{theorem}{section}
\numberwithin{equation}{section}

\begin{document}
\title{\bf\Large Quantitative Bounds and Compactness for the  Commutators of Area Integrals Associated with Self-adjoint Operators on Weighted $L^p$ and Morrey Spaces
	
\footnotetext{\hspace{-0.35cm} 2010 {\it Mathematics Subject Classification}. 42B20, 42B25, 42B35, 47B47.\endgraf
{\it Key words and phrases:}
Area integrals, Commutators, Self-adjoint operator, Quantitative weighted estimates, Compactness, Morrey space.\endgraf
}

\author{\smallskip
Chunmei Zhang and  
Xiangxing Tao 
\endgraf
 }
}
\date{}
\maketitle
	
\vspace{-0.8cm}

\begin{center}
\begin{minipage}{13cm}
{\small {\bf Abstract}\quad
Let $L$ be a non-negative self-adjoint operator, we consider some commutators generated by the BMO function $b$ and the area integral operator $S_H$ associated with the heat semigroup $\{e^{-tL}\}_{t>0}$ or the area integral operator $S_P$ associated with the Poisson semigroup $\{e^{-t\sqrt{L}}\}_{t>0}$. The strong-type estimates of these commutators on weighted $L^p$ spaces and weighted Morrey spaces are established.
At the same time, we verified that these commutators are compact operators on weighted Morrey spaces.
}
\end{minipage}
\end{center}
	
\vspace{0.2cm}
	
%\tableofcontents

\vspace{0.2cm}

\section{Introduction and main results}
 Since the well-known Kato conjecture was solved in 2002 by Auscher, Hofmann, Lacey et al.\cite{AHLMT}, interest has grown in developing a new Calder\'on-Zygmund theory for operators that lie beyond the classical framework, since their kernels do not possess the required decay or smoothness.

In this paper, we suppose that $L$ is any non-negative self-adjoint operator on $L^2\left(\mathbb{R}^n\right)$ and that the heat semigroup $e^{-t L}$, generated by $-L$ on $L^2\left(\mathbb{R}^n\right)$, has the kernel $h_t(x, y)$  satisfying the following Gaussian upper bound,
\begin{equation}\label{(GE)}
	\qquad \qquad \left|h_t(x, y)\right| \leq \frac{C}{t^{n / 2}} \exp \left(-\frac{|x-y|^2}{c t}\right)	
\end{equation}
for all $t>0$ and $x, y \in \mathbb{R}^n$, where $C$ and $c$ are positive constants.
Such estimates are typical for elliptic differential operator of order two, Laplace operator on an open connected domain with Dirichlet boundary conditions, Schr\"odinger operator with a nonnegative potential, and so on.
We  consider the area integral $S_H$ associated with the heat semigroup $\{e^{-t L}\}_{t>0}$, and the area integral $S_P$ associated with the Poisson semigroup $\{e^{-t\sqrt{L}}\}_{t>0}$, which are defined respectively as, \begin{equation}\label{eq2}
	S_H f(x)=\left(\int_{0}^{\infty}\int_{|x-y|<t}\left|t^2 L e^{-t^2 L} f(y)\right|^2 \frac{d y d t}{t^{n+1}}\right)^{1 / 2},
\end{equation}
% and
\begin{equation}\label{eq1}
	S_P f(x)=\left(\int_{0}^{\infty}\int_{|x-y|<t}\left|t \sqrt{L} e^{-t \sqrt{L}} f(y)\right|^2 \frac{d y d t}{t^{n+1}}\right)^{1 / 2},
\end{equation}
initially for $f \in \mathcal{S}\left(\mathbb{R}^n\right)$, the class of Schwartz functions.

It is acknowledged that when $L=-\Delta$ on $\mathbb{R}^n$, $S_P$ and $S_H$ reduce to the classical Lusin area integrals which play an important role in harmonic analysis and PDE, see \cite{fs24} for more details. 
Area integrals associated to abstract operators are also crucial in harmonic analysis. Using $S_P$ and $S_H$, Auscher et al. \cite{Aus} introduced the Hardy space $H_L^1$ adapted to the operator $L$. Subsequently, Duong and Yan \cite{TY} proved that $\BMO_{L^*}$($L^*$ is the adjoint
operator of $L$) is the dual space of the
Hardy space $H_L^1$, which  generalized Fefferman and Stein's result on the duality between
$H^1$ and $\BMO$ spaces. Later on, the theory of function spaces associated with operators has been developed and
generalized to many other different settings, see for example  \cite{DL10,HLM,HM,SY}. Recently, Martell and Prisuelos-Arribas \cite{MP1} studied weighted norm inequalities for area integrals, establishing the boundedness of various square functions in weighted Lebesgue spaces via heat and Poisson semigroups. Furthermore, in \cite{MP2}, they utilized these square functions to introduce corresponding weighted Hardy spaces $H_L^1(w)$ and demonstrated their equivalence, which follows from the comparability of the square functions in the relevant weighted spaces.
For the second-order divergence form operator $L$, Hofmann and Mayboroda \cite{HM} verified that the Hardy space $H_L^1$ can be characterized by the square functions $S_H$ and $S_P$. Building on this work, Hofmann, Lu et al. \cite{HLM} obtained a similar characterization for Hardy spaces associated with operators satisfying Gaussian or Davies-Gaffney estimates, using the functions $S_H$ and $S_P$ defined in \eqref{eq2} and \eqref{eq1}. For more details, we refer to \cite{Aus,DL10,SY} and references therein.

For $1<p<\infty$ and $w\in A_p$ (the Muckenhoupt class of weights), consider a non-negative self-adjoint operator $L$ whose heat kernel satisfies the Gaussian bound \eqref{(GE)}. For the square functions $S_H$ and $S_P$ associated with $L$, the weighted $L^p(w)$ estimate was established by Martell and Prisuelos-Arribas \cite{MP1} and Gong and Yan \cite{GY} as follow,
\begin{equation}\label{eq1.1}
	\left\|S_P f\right\|_{L^{p}(w)}+	\left\|S_H f\right\|_{L^{p}(w)} \leq C\|f\|_{L^{p}(w)},
\end{equation}
for all $f \in L^{p}(w)$. 
More general, the boundedness of $S_P$ and $S_H$ on weighted Morrey spaces were obtained by the Gong \cite{Gong}, that is,
\begin{equation}\label{eq1.3}
	\left\|S_P f\right\|_{L^{p, \kappa}(w)}+\left\|S_H f\right\|_{L^{p, \kappa}(w)} \leq C\|f\|_{L^{p, \kappa}(w)}
\end{equation}
for all $f \in L^{p, \kappa}(w)$ and $0<\kappa<1$. Here the weighted Morrey space $L^{p, \kappa}(w)$ is defined by
$$
L^{p, \kappa}(w):=\left\{f \in L_{\rm{loc}}^p(w):\|f\|_{L^{p, \kappa}(w)}<\infty\right\}
$$
for $1 \leq p<\infty, 0<\kappa<1$  and the weight $w$, 
where
$$
\|f\|_{L^{p, \kappa}(w)}=\sup _B\left(\frac{1}{w(B)^\kappa} \int_B|f(x)|^p w(x) d x\right)^{1 / p},
$$
and the supremum is taken over all balls $B$ in $\mathbb{R}^n$. Particularly, if $\kappa=\lambda / n$ with $0<\lambda<n$, then $L^{p, \kappa}(w)=L^{p, \lambda}\left(w\right)$ means the classical weighted Morrey spaces.

On the other hand, the commutator, first introduced by Calder\'on \cite{Ca}, is a fundamental tool in harmonic analysis and PDEs. It is particularly important in theories such as that of nondivergent elliptic equations with discontinuous coefficients (see \cite{Ch,fm,fr}). In a later work, Coifman, Rochberg, and Weiss \cite{CRW} established that for $b\in \BMO(\mathbb R^n)$ and a standard Calder'on-Zygmund singular integral operator $T_{cz}$, the commutator defined by
$$[b,T_{cz}](f)(x)=b(x)T_{cz}f(x)-T_{cz}(bf)(x)$$
is bounded on
$L^p(\mathbb R^n)$ for $1<p<\infty.$  Here, $\BMO$ space is the set of all locally integral functions $b$ satisfying $$\|b\|_{\BMO}:=\sup_Q\frac{1}{|Q|}\int_Q|b(y)-b_Q|dy<\infty,$$
where $b_Q:=\frac{1}{|Q|}\int_Qb(t)dt$ and the supremum is taken over all cubes $Q$ in $\mathbb R^n$ with sides parallel to the coordinate axes. In 1995, P\'erez \cite{P95} pointed out that $[b,T_{cz}]$ fail to be 
of weak type (1,1) and he provided a substitution of the weak-type boundedness by a kind of weak $L\log L$-type
estimates. 
Subsequently, significant progress has been made in understanding the strong and weak type estimates for commutators of various operators, including singular integrals with rough kernels, multilinear Calder\'on-Zygmund operators and Littlewood-Paley square functions. For example, Xue and Ding {\cite{DX} established the boundedness of the commutators of Lusin area integrals. Recently,
	Cao, Si and Zhang \cite{CSZ} investigated several kinds of weighted norm inequalities for such operators and they consider other weak type
	estimates including the restricted weak-type $(p,p)$ estimates and the endpoint estimate for the corresponding commutators. We refer the readers to \cite{BHS,Ch,LOR17,LOR21,[PT]} and the references therein. 
	
	The first purpose of the paper is to investigate the following commutators related to the area integrals $S_{P}$ and $S_{H}$, 
	$$
	\begin{aligned}
		&S_{P,b}f(x):=\left(\int_{0}^{\infty}\int_{|x-y|<t}\left|t \sqrt{L} e^{-t \sqrt{L}} ((b(x)-b(\cdot))f)(y)\right|^2 \frac{d y d t}{t^{n+1}}\right)^{1 / 2}, \\
		&S_{H,b}f(x):=\left(\int_{0}^{\infty}\int_{|x-y|<t}\left|t^2 L e^{-t^2 L}((b(x)-b(\cdot)) f)(y)\right|^2 \frac{d y d t}{t^{n+1}}\right)^{1 / 2},
	\end{aligned}
	$$ 
	initially for $f \in \mathcal{S}\left(\mathbb{R}^n\right)$.
	
	We will establish the quantitative weighted estimates for the commutators $S_{P,b}$ and  $S_{H,b}$ in weighted $L^p(w)$ spaces with $w\in A_p$ as follows.
	\medspace
	\begin{theorem}\label{lebesgue bound}
		Let $L$ be a non-negative self-adjoint operator and the corresponding heat kernel satisfies Gaussian estimates \eqref{(GE)}. Assume that $1<p<\infty$ and $w \in A_p$, $b\in \BMO(\rn)$, then there exists a constant $C$, such that
		\begin{equation}\label{lp}
			\left\|S_{P,b}f\right\|_{L^{p}(w)} \leq C\|b\|_{\BMO}[w]_{A_p}^{\alpha_p+\frac{1}{p-1}}\|f\|_{L^p(w)}
		\end{equation}
		for all $f \in L^{p}(w)$, where $\alpha_p=\max\{\frac{1}{2},\frac{1}{p-1}\}$. Also, estimate \eqref{lp} holds for the commutator $S_{H,b}$.	
	\end{theorem}

	Furthermore, we extended the results in Theorem \ref{lebesgue bound} to the weighted Morrey spaces, which can be stated as follows.
	
	%We established the following estimates for the commutators of these area integrals on weighted Morrey spaces, which will genralize the results in Theorem \ref{lebesgue bound}.
	
	\begin{theorem}\label{morreybound}
		Let L be a non-negative self-adjoint operator and the corresponding heat kernel satisfies Gaussian bounds \eqref{(GE)}. Assume that $1<p<\infty$ and $0<\kappa<1$. If $w \in A_p$ and $b\in {Osc_{\exp L^r}},\ r\geq1$, then there exists a constant $C$ such that
		\begin{equation}\label{bound}
			\left\|S_{P,b}f\right\|_{L^{p, \kappa}(w)} \leq C \|b\|_{Osc_{\exp L^r}}\|f\|_{L^{p, \kappa}(w)}
		\end{equation}
		for all $f \in L^{p, \kappa}(w)$. Also, estimate \eqref{bound} holds for the commutator  $S_{H,b}$.	
	\end{theorem}
 
	{\remark
		The definition of Orlicz space ${Osc_{\exp L^r}}$ will be presented in Section \ref{De}. To prove Theorem \ref{morreybound}, we will  introduce an auxiliary square function $g_{\mu, \Psi}^*$ which can dominated the area integrals above and establish the related pointwise estimate through sharp maximal functions. }
	
	\vspace{0.2cm}
	Obseving that $Osc_{\exp L^1}$ coincides with $\BMO(\rn)$,  %(see \cite[p.675]{[PT]})
	from Theorem \ref{morreybound}, we establish the quantitative weighted estimates for the commutators $S_{P,b}$ and  $S_{H,b}$ with $b\in \BMO(\rn)$.
	
	\medskip
	\begin{corollary}\label{morreybound-bmo}
		Let $L$ be a non-negative self-adjoint operator, such that the corresponding heat kernels satisfy Gaussian bounds \eqref{(GE)}. Let $1<p<\infty$ and $0<\kappa<1$. If $w \in A_p$ and $b\in \BMO $, then there exists a constant $C$ such that
		\begin{equation}\label{bmobound}
			\left\|S_{P,b}f\right\|_{L^{p, \kappa}(w)} \leq C \|b\|_{\BMO}\|f\|_{L^{p, \kappa}(w)},
		\end{equation}
		for all $f \in L^{p, \kappa}(w)$. Also, estimate \eqref{bmobound} holds for the commutator $S_{H,b}$.	
	\end{corollary}

	\medskip
	%Different from classical methods, 
	
	Our second purpose of the paper is to study the compactness of the commutators $S_{P,b}$ and $S_{H,b}$.  An operator $T: X \to Y$ between Banach spaces is said to be compact if it is continuous and maps bounded sets to relatively compact sets. The study of  compactness for commutators originated from Uchiyama \cite{Uch}, who characterized the compactness of $[b, T_\Omega]$ on $L^p(\mathbb{R}^n)$ by the condition $b \in \text{CMO}(\mathbb{R}^n)$, where $\text{CMO}(\mathbb{R}^n)$ is the closure of $C_c^\infty(\mathbb{R}^n)$ in $\text{BMO}(\mathbb{R}^n)$. 
	This line of research has since expanded to various operators and function spaces, including singular integrals with variable kernels \cite{CD1}, Marcinkiewicz integrals on Morrey spaces \cite{CDW1}, and parametric area integrals \cite{CW15}, with $\text{CMO}$ remaining the characterizing condition. Further developments and applications, including extrapolation techniques for compact operators, can be found in \cite{BDMT2, DM15, TXYY, TX, XZ, COY1, H1, H3}.

	In order to give the compactness of  the commutators $S_{P,b}$ and  $S_{H,b}$, we assume that the heat kernel   $h_t(x, y)$ of the semigroup $e^{-t L}$ also satisfies the following regularity estimates: there exist positive constants $C, c \in(0, \infty)$ and $\nu \in(0,1]$ such that for all $t \in(0, \infty)$ and almost every $x, y, z \in \mathbb{R}^n$ with $2|z| \leq t^{1 / 2 }+|x-y|$, it holds that
	\begin{equation}\label{smooth}
		\begin{aligned}
			& \left|h_t(x+z, y)-h_t(x, y)\right|+\left|h_t(x, y+z)-h_t(x, y)\right| \\
			\vspace{0.2cm}
			& \quad \leq \frac{C}{t^{n / 2 }}\Big(\frac{|z|}{t^{1 / 2}+|x-y|}\Big)^\nu \exp \Big\{-\frac{c|x-y|^{2}}{t}\Big\}.
		\end{aligned}
	\end{equation}

\medskip
The assumption \eqref{smooth} is reasonable. In fact, let $L=-{\rm div}(A\nabla)$ be the divergence form elliptic operator in $L^2(\rn)$, where $A$ has real entries when $n\geq3$ and complex entries when $n \in\{1, 2\}$. From \cite[Chapter 1]{AT} or \cite{BCD}, we can see that the heat kernel associated with the operator
$L$ satisfies the assumptions (\ref{(GE)}) and (\ref{smooth}). 
Moreover, we will prove that $S_{P,b}$ and $ S_{H,b}$ enjoy the following  property of compactness.

\medskip
\begin{theorem}\label{compact thm}
	Let $L$ be a non-negative self-adjoint operator such that the corresponding heat kernel $h_t$ satisfies \eqref{(GE)} and \eqref{smooth}, then for $b \in \operatorname{CMO}\left(\mathbb{R}^n\right)$, commutators $S_{P,b}$ and $ S_{H,b}$ are all compact operators in $L^{p, \lambda}\left(w\right)$ for $0<\lambda<n, 1<$ $p<\infty$.	
\end{theorem}

The organization of this article is as follows. In Section \ref{De}, we prepare
some necessary definitions and lemmas. Section \ref{proof}  will devote to establishing the weighted boundedness for the commutators of 
$S_P$ and $S_H$ in Lebesgue spaces and Morrey spaces via applying  Chauchy integral and  sharp maximal operators to an important auxiliary function.
Based on weighted estimates of $S_P$ and $S_H$, we will give the proof of Theorem \ref{compact thm} via Frechet-Komogorov theorem adapted to weighted Morrey spaces and smooth truncated techniques. In what follows, the character  $C$ or $c$, sometimes with certain parameters, always means a
positive constant that is independent of the main parameters
involved but whose value may differ from line to line.

\section[Definitions and some lemmas]{Definitions and some lemmas}\label{De}
\indent Recall that, if $L$ is a positive definite self-adjoint operator acting on $L^2(\rn)$, then it adimits a spectral resolution
$$
L=\int_0^{\infty} \lambda d E(\lambda).
$$
For every bounded Borel function $F : [0, \infty) \to C$, by using the spectral 
theorem, we can define the operator 
$$F(L) :=\int_{0}^{\infty}F(\lambda)d E_L(\lambda).$$

The following results are useful for the proof of our main theorems.

\begin{lemma}[{\cite[Lemma 2.1]{GX}}]\label{kernel}
	Let $\varphi \in C_0^{\infty}(\mathbb{R})$ be even, $\operatorname{supp} \varphi \subset(-1,1)$ and denote  $\Phi$ the Fourier transform of $\varphi$. Then for every $k=0,1,2, \ldots$, and for every $t>0$, the kernel $K_{\left(t^2 L\right)^k \Phi(t \sqrt{L})}(x, y)$ of the operator $\left(t^2 L\right)^k \Phi(t \sqrt{L})$, which was defined by the spectral theory, satisfies
	
	\begin{equation}\label{kerelsupp}
		\operatorname{supp} K_{\left(t^2 L\right)^k \Phi(t \sqrt{L})} \subseteq\left\{(x, y) \in \mathbb{R}^n \times \mathbb{R}^n:|x-y| \leq t\right\},
	\end{equation}	
	and
	
	$$
	\left|K_{\left(t^2 L\right)^\kappa \Phi(t \sqrt{L})}(x, y)\right| \leq C t^{-n},
	$$
	
	\noindent	for all $t>0$ and $x, y \in \mathbb{R}^n$.
\end{lemma}

\smallskip
In this paper, we  will work with the following  Muckenhoupt {$A_p$} weights.

\begin{definition}[\textbf{\rm{$A_p$} weights}]
	Suppose that $\omega$ is a nonnegative locally integrable function defined on $\rn$.
	\begin{itemize}
		\item[\rm(i)]
		Let $\ 1<p<\infty,$ we call $\omega$ belongs to the Muckenhoupt weight class $A_p$, if
		$$[\omega]_{A_p}:=\sup_{Q \subset \rn}\left(\frac{1}{|Q|}\int_Q\omega(t)dt\right)\left(\frac{1}{|Q|}\int_
		Q\omega^{-\frac{1}{p-1}}(t)dt\right)^{p-1}<\infty;$$
		\item[\rm(ii)]We call  $\omega$ belongs to the Muckenhoupt weight class $A_1$, if
		$$[\omega]_{A_1}:=\displaystyle\sup_{Q \subset \rn}\Big(\frac{1}{|Q|}\int_{Q}\omega(t)dt\Big)\|\omega^{-1}\|_{L^\infty(Q)}<\infty;$$
		\item[\rm(iii)]We call $\omega$ belongs to the Muckenhoupt weight class $A_\infty$, if
		$$[\omega]_{A_\infty}:=\displaystyle\sup_{Q \subset \rn}\Big(\frac{1}{|Q|}\int_{Q}\omega(t)dt\Big)\exp\Big(\frac{1}{|Q|}\int_{Q} \log
		\omega(t)^{-1}dt\Big)<\infty.$$
	\end{itemize}
\end{definition}

\medskip
It was proved that $A_p$ weights class enjoy the following properties.

\begin{lemma}[{\cite[Lemma 2.1]{XZ}}]\label{weights}
	Let $w\in A_p,\ 1<p<\infty,$ then
	\begin{enumerate}%[{\rm (a)}]
		\item  [\rm(a)] for any ball $B$ and $\lambda>1$, we have
		$$w(\lambda B)\leq \lambda^{np}[w]_{A_p}w(B),$$
		where $w(B)=\int_B w(x) dx;$
		
		\vspace{0.1cm}
		\item [\rm(b)]
		$A_p=\mathop\bigcup\limits_{1<q<p}A_q;$
		\item  [\rm(c)] there exists a constant $\varepsilon>0$ depending only on $n,\ p$ and $[w]_{A_p}$ such that $w^{1+\varepsilon}\in A_p$ and $$[w^{1+\epsilon}]_{A_p}\leq C ([w]_{A_p})^{1+\varepsilon}.$$
	\end{enumerate}
\end{lemma}

\medskip
Recall that the $A_p$ weight class satisfies a reverse H\"older property, the precise statement is as follows,  which can be found in \cite{P13}.

\medskip
\begin{lemma}[{\cite[Lemma 8.1]{P13}}]\label{reverse-Holder} 
	Let $1<p<\infty, w \in A_p$ and
	$$
	r_w=1+\frac{1}{2^{n+2 p+1}[w]_{A_p}} .
	$$
	Then
	$$
	\left(\frac{1}{|Q|} \int_Q w^{r_w}(x) d x\right)^{1 / r_w} \leq \frac{2}{|Q|} \int_Q w(x)dx.
	$$
\end{lemma}

\medskip

By John-Nirenberg inequality, Chung, Pereyra and Per\'ez \cite{CPP12} gave the following property of $\rm BMO$ functions.

\medskip
\begin{lemma}[{\cite[Lemma 2.2]{CPP12}}]\label{JN3} 
	Suppose that $1 <p<\infty$ and $b\in \BMO(\rn)$, then there exist $\gamma_n\in(0,1),\ \zeta_n\in(0,\infty)$ depending only on $n$ such that for any $\delta\in\mathbb R$ with $|\delta|\leq \frac{\gamma_n}{\|b\|_{\BMO}}\min\{1,\frac{1}{p-1}\}$, it holds 
	\begin{center}
		$e^{\delta b(x)}\in A_p$ and $\big[e^{\delta b}\big]_{A_p}\leq\zeta_n^p.$
	\end{center}
\end{lemma}

\medskip
In order to prove the boundedness of $S_{P,b}$ and $S_{H,b}$, we need to introduce the Fefferman-Stein sharp maximal functions.   
For $\delta>0$, define
%\subsection{The Sharp  Maximal Function}
$$M_\delta f(x)=\big[M(|f|^\delta)(x)\big]^\frac{1}{\delta}=\Big(\displaystyle\sup_{Q\ni x}\frac{1}{|Q|}\int_Q |f(y)|^\delta dy\Big)^\frac{1}{\delta}.$$
If $\delta=1$, $M_1f$ will be the classical Hardy-Littlewood maximal function $Mf$. 
The Fefferman-Stein sharp maximal function $M^{\sharp}$ is defined as
$$M^\sharp f(x)=\displaystyle\sup_{Q\ni x}\inf_c \frac{1}{|Q|}\int_Q|f(y)-c|dy\approx \displaystyle\sup_{Q\ni x}\frac{1}{|Q|}\int_Q|f(y)-\langle f\rangle_Q|dy,$$ and denote
$$M_\delta^{\sharp}f(x)=M^{\sharp}(|f|^\delta)(x)^\frac{1}{\delta},$$
where $\langle f\rangle_Q=\frac{1}{|Q|}\int_Qf(x)dx$ represents the integral average of $f$ over the cube $Q$.

The relationship between $M_\delta$ and $M_\delta ^\sharp$ is given by the next lemma.

\begin{lemma}[\cite{[PT]}]\label{lemwi}
	Let $\omega\in A_\infty$, there exists a positive constant $C$ that depends on the $A_\infty$ condition of $\omega$ such that for all $\la,\varepsilon>0$,
	$$\omega(\{y\in{\rn}:Mf(y)>\la,\ M^{\sharp}f(y)\leq \la \varepsilon\})\leq C \varepsilon^\rho \omega(\{y\in{\rn}:Mf(y)>\frac{\la}{2}\}).$$
	Thus, for $\delta>0$, we can get the following {estimates}:\\
	{\rm (i)} Suppose $\va:(0,\infty)\rightarrow(0,\infty)$ is doubling, then there exists a positive constant $C$ that depends on the $A_\infty$ condition of $\omega$ and the doubling condition of $\va$ such that
	$$\displaystyle\sup_{\la>0}\va(\la)\omega(\{y\in{\rn}:M_\delta f(y)>\la\})\leq C \displaystyle\sup_{\la>0}\va(\la)\omega(\{y\in{\rn}:M_\delta^{\sharp}f(y)>\la\})$$ holds for every function $f$ that makes the left-hand side of the inequality finite.\\
	{\rm (ii)} For $0<p<\infty$, there exists a positive constant $C$ that depends on the $A_\infty$ condition of $\omega$ and $p$ such that$$\int_{\rn}\big(M_\delta f(x)\big)^p\omega(x)dx\leq C\int_{\rn}\big(M_\delta^{\sharp} f(x)\big)^p\omega(x)dx$$holds for every function $f$ that makes the left-hand side of the inequality finite.
\end{lemma}

\medskip
Next, we introduce the Orlicz spaces.

\medskip
\begin{definition}[{\bf Orlicz maximal function}]
	We call $\Phi:[0,\infty)\rightarrow [0,\infty)$ is  a Young function, if $\Phi$ is a continuous, convex and increasing function such that $$\displaystyle\lim_{t\rightarrow 0^+} \frac{\Phi(t)}{t}=\displaystyle\lim_{t\rightarrow\infty}\frac{t}{\Phi(t)}=0.$$
	The $\Phi$-average of a function $f$ over a cube $Q$ is defined as follows,
	$$\|f\|_{\Phi,Q}=\inf\{\la>0:\frac{1}{|Q|}\int_Q\Phi(\frac{|f(x)}{\la})dx\leq 1\}.$$
	The maximal function about $\Phi$ is defined as $$M_\Phi(f)(x)=\sup_{x\in Q}\|f\|_{\Phi,Q},$$where the supremum is taken over all the cubes containing $x$.
\end{definition}

\medskip
It's easy to observe that 
\iffalse
$$\|f\|_{\Phi,Q}>1\quad \text{if and only if}\quad \frac{1}{|Q|}\int_Q\Phi(|f(x)|)dx>1.$$
What's more, 
\fi
if $\Phi_1$ and $\Phi_2$ are Young functions with $\Phi_1(t)\leq \Phi_2(t)$ when $t\geq t_0>0$, then
$$\|f\|_{\Phi_1,Q}\leq C\|f\|_{\Phi_2,Q}.$$
In particular, if $\Phi(x)=e^{x^r}-1$, then we write $\|f\|_{\Phi,Q}=\|f\|_{\exp L^r,Q},$\ $M_\Phi(f)(x)=M_{\exp L^r}(f)(x)$. If $\Phi(x)=x(1+\log^+x)^r$, then we write $\|f\|_{\Phi,Q}=\|f\|_{L(\log L)^r,Q},\ M_\Phi(f)(x)=M_{L(\log L)^r}(f)(x)$. Notice that for any $r>0$, $M(f)\leq M_{L(\log L)^r}(f)$ and for any $k\in \mathbb{N}$, $M_{L(\log L)^k}\sim M^{k+1}$, where $M$ denotes the Hardy-Littlewood maximal function.
 
\begin{definition}[{\bf Orlicz space}]
	For $r\geq 1$, $Osc_{\exp L^r}$ space is defined as
	$$Osc_{\exp L^r}=\{f\in L_{\rm loc}^1(\rn):\|f\|_{Osc_{\exp L^r}}<\infty\},$$
	where $$\|f\|_{Osc_{\exp L^r}}=\sup_Q\|f-f_Q\|_{\exp L^r,Q}=\sup_Q\|f-f_Q\|_{e^{t^r}-1,Q}.$$
\end{definition}

\medskip
According to \cite{[PT]}, we can obtain $Osc_{\exp L^1}=\BMO(\rn)$, $Osc_{\exp L^r}\subsetneq \BMO(\rn)$
and $\|b\|_{\BMO}\leq C\|b\|_{Osc_{\exp L^r}}$ when $r>1$ .
Associated to each Young function $\Phi$, one can define a complementary function
$$\bar{\Phi}(s)=\sup_{t>0}\{st-\Phi(t)\}.$$
It's clear that $\bar{\Phi}$ is also a Young function. As a consequence, the following H\"older's inequality holds in Orlicz spaces,
\begin{equation}\label{Holder}
	\frac{1}{|Q|}\int_Q|f(x)g(x)|dx\leq 2\|f\|_{\Phi,Q}\|g\|_{\bar{\Phi},Q}.
\end{equation}
For more related results about Orlicz spaces, one can refer to \cite{O,RR}.

Finally, the following result of $M$ on weighted Morrey spaces is useful in our proof of main theorems.

\begin{lemma}[\cite{CF87}]\label{maximal}
	For $1<p<\infty, 0<\kappa<1$, and $w \in A_p$, we have 
	$$\|M f\|_{L^{p, \kappa}(w)} \leq C\|f\|_{L^{p, \kappa}(w)}.$$
\end{lemma}

	\section{Proof of Theorems }\label{proof}

To prove the main theorems, let us introduce an auxiliary $g_{\mu, \Psi}^*$ function. 	
Let $\varphi \in C_0^{\infty}(\mathbb{R})$ be even function with $\int_{\rn} \varphi(x)dx=1, \operatorname{supp} \varphi \subset(-1 / 10,1 / 10)$. Let $\Phi$ denote the Fourier transform of $\varphi$ and let $\Psi(s)=s^{2 n+2} \Phi^3(s)$. We define the $g_{\mu, \Psi}^*$ function by
$$
g_{\mu, \Psi}^*(f)(x)=\left(\iint_{\mathbb{R}_{+}^{n+1}}\left(\frac{t}{t+|x-y|}\right)^{n \mu}|\Psi(t \sqrt{L}) f(y)|^2 \frac{d y d t}{t^{n+1}}\right)^{1 / 2}, \quad \mu>1 .
$$	
Similarly, we can define the commutator of $g_{\mu, \Psi}^*$ as follows,
$$
g_{\mu, \Psi,b}^*(f)(x)=\left(\iint_{\mathbb{R}_{+}^{n+1}}\left(\frac{t}{t+|x-y|}\right)^{n \mu}\big|\Psi(t \sqrt{L})((b(x)-b(\cdot)) f)(y)\big|^2 \frac{d y d t}{t^{n+1}}\right)^{1 / 2}, \quad \mu>1 .
$$		

According to \cite{GY}, the $g_{\mu, \Psi}^*$ function can dominate the corresponding area integrals and is bounded on $L^p(w)$.
\begin{proposition}[{\cite[Propositions 3.3 and 3.4]{GY}}]
	Let $L$ be a non-negative self-adjoint operator, such that the corresponding heat kernels satisfy condition \eqref{(GE)}. Then for $f \in \mathcal{S}\left(\mathbb{R}^n\right)$, there exists a constant $C=C_{n, \mu, \Psi}$, such that the area integral $S_P$ satisfies the pointwise estimate		
	\begin{equation}\label{eq3.2}
		S_P f(x) \leq C g_{\mu, \Psi}^*(f)(x) .
	\end{equation}	
	Estimate \eqref{eq3.2} also holds for the area integral $S_H$.	
\end{proposition}	

\begin{lemma}[{\cite[Theorem 1.4]{GY}}]\label{upperbound}
	Let $L$ be a non-negative self-adjoint operator such that the corresponding heat kernels satisfy Gaussian bounds \eqref{(GE)} and $T$ be one of the area functions $S_P, S_H$ and $g_{\mu, \Psi}^*$ with $\mu>3$. If $1<p<\infty$ and $w \in A_p$, then there exists a constant $C$, such that
	\begin{equation}\label{wb}
		\left\|Tf\right\|_{L^{p}(w)} \leq C[w]_{A_p}^{\alpha_p+1/(p-1)}\|f\|_{L^p(w)},
	\end{equation}
	where $\alpha_p=\max\{\frac{1}{2},\frac{1}{p-1}\}$.	
\end{lemma}

Gong considered the boundedness of $g_{\mu, \Psi}^*$ on $ L^{p, \kappa}(w)$ in \cite{Gong}.

\begin{lemma}[\cite{Gong}]\label{th3.2}
	Let L be a non-negative self-adjoint operator, such that the corresponding heat kernels satisfy Gaussian bounds \eqref{(GE)}. Let $\mu>3$, $1<p<\infty$, and $0<\kappa<1$. If $w \in A_p$, then there exists a constant $C$ such that
	$$
	\left\|g_{\mu, \Psi}^* f\right\|_{L^{p, \kappa}(w)} \leq C\|f\|_{L^{p, \kappa}(w)},
	$$
	for all $f \in L^{p, \kappa}(w)$.
\end{lemma}

%\medskip
Now we are at the position to prove our main theorems. Due to \eqref{eq3.2}, in order to prove Theorem \ref{lebesgue bound} and Theorem \ref{morreybound}, it suffices to show the corresponding results for $g_{\mu, \Psi,b}^*$.

%\medskip
\begin{proof}[\bf Proof of Theorem \ref{lebesgue bound}] 
	
	Since $w\in A_p$,	by (c) of Lemma \ref{weights}, we can find
	an $\epsilon>0$ such that $w^{1+\epsilon}\in A_p$. So by the assumption of induction, it holds that
	\begin{equation}\label{f1}
		\|g_{\mu, \Psi}^*(\varphi)\|_{L^p(w^{1+\epsilon})}\leq C_{p,n,\epsilon}[w]_{A_p}^{(1+\epsilon)(\alpha_p+1/(p-1))}\|\varphi\|_{L^p(w^{1+\epsilon})}, \quad \varphi\in L^p(w^{1+\epsilon}) 
	\end{equation}
	for $\mu>3$.
	Now we take $\delta={p\gamma\cos\theta(1+\epsilon)}/{\epsilon}$ in Lemma \ref{JN3} and choose $\gamma=\frac{\epsilon\alpha_n}{p(1+\epsilon)\| b\| _{\BMO}}\min\{1, \frac{1}{p-1}\}$ such that $ |\delta|\leq \frac{\alpha_n}{\| b\| _{\BMO}}\min\{1, \frac{1}{p-1}\}$, then by Lemma \ref{JN3}, $e^{{pb\gamma\cos\theta(1+\epsilon)}/{\epsilon}}\in A_p$ and
	$$\big[e^{{pb\gamma\cos\theta(1+\epsilon)}/{\epsilon}}\big]_{A_p}\leq\beta_n^p.$$

	By the weighted boundedness of $g_{\mu, \Psi}^*$, we obtain that for any $\theta\in[0,2\pi],\
	\varphi\in L^p(e^{{pb\gamma\cos\theta(1+\epsilon)}/{\epsilon}}),$
	\begin{equation}\label{f2}
		\|g_{\mu, \Psi}^*(\varphi)\|_{L^p(e^{{pb\gamma\cos\theta(1+\epsilon)}/{\epsilon}})}\leq C_{n,p}[e^{{pb\gamma\cos\theta(1+\epsilon)}/{\epsilon}}]_{A_p}^{\alpha_p+1/(p-1)}
		\|\varphi\|_{L^p(e^{{pb\gamma\cos\theta(1+\epsilon)}/{\epsilon}})}.
	\end{equation}
	Applying the Stein-Weiss interpolation theorem with change of measures \cite{SW58} between
	(\ref{f1}) and (\ref{f2}), we have for any $\theta\in [0, 2\pi]$ and $\varphi\in L^p(w e^{pb\gamma\cos\theta})$,
	\begin{equation*}
		\|g_{\mu, \Psi,b}^*(\varphi)\|_{L^p(w e^{{pb\gamma\cos\theta}})}\leq C_{n,p}
		(\beta_n^{{p\epsilon}/{(1+\epsilon)}}[w]_{A_p})^{\alpha_p+1/(p-1)}\|\varphi\|_{L^p(w e^{{pb\gamma\cos\theta}})}.
	\end{equation*}

	Denote $F(z)=e^{z(b(x)-b(y))},\ z\in \mathbb C.$ Then by the analyticity of $F(z)$ on $\mathbb C$ and the Cauchy integration formula, we have
	$$b(x)-b(y)=F'(0)=\frac{1}{2\pi i}\int_{|z|=\gamma}\frac{F(z)}{z^2}dz=\frac{1}{2\pi \gamma}\int_0^{2\pi}e^{\gamma e^{i\theta}(b(x)-b(y))}e^{-i\theta}d\theta.$$
	Therefore by Minkowski's inequality, we have
	\begin{align*}
		&g_{\mu, \Psi,b}^* (f)(x)\\
		&=\left(\iint_{\mathbb{R}_{+}^{n+1}}\left(\frac{t}{t+|x-y|}\right)^{n \mu}\big|\Psi(t \sqrt{L})((b(x)-b(\cdot)) f)(y)\big|^2 \frac{d y d t}{t^{n+1}}\right)^{1 / 2}\\
		&=	\left(\iint_{\mathbb{R}_{+}^{n+1}}\left(\frac{t}{t+|x-y|}\right)^{n \mu}\Big|\int_{\mathbb R^n}\Big(\frac{1}{2\pi \gamma}\int_0^{2\pi}e^{\gamma e^{i\theta}(b(x)-b(y))}e^{-i\theta}d\theta\Big)K_{\Psi(t \sqrt{L})}(x,y)f(y)dy \Big|^2 \frac{d y d t}{t^{n+1}}\right)^{1 / 2}\\
		&=\frac{1}{2\pi \gamma}	\left(\iint_{\mathbb{R}_{+}^{n+1}}\left(\frac{t}{t+|x-y|}\right)^{n \mu}\Big|\int_0^{2\pi}\int_{\mathbb R^n}e^{-\gamma e^{i\theta}b(y)}K_{\Psi(t \sqrt{L})}(x,y)f(y)dye^{\gamma e^{i\theta} b(x)}e^{-i\theta}d\theta \Big|^2 \frac{d y d t}{t^{n+1}}\right)^{1 / 2}\\
		&\leq \frac{1}{2\pi\gamma}\int_0^{2\pi}	\left(\iint_{\mathbb{R}_{+}^{n+1}}\left(\frac{t}{t+|x-y|}\right)^{n \mu}\Big|\int_{\mathbb R^n}K_{\Psi(t \sqrt{L})}(x,y)e^{-\gamma e^{i\theta}b(y)}f(y)dy\Big|^2 \frac{d y d t}{t^{n+1}}\right)^{1 / 2}|e^{\gamma e^{i\theta} b(x)}|d\theta\\
		&=\frac{1}{2\pi\gamma}\int_0^{2\pi}g_{\mu, \Psi}^* (fe^{-\gamma e^{i\theta}b(\cdot)})(x)e^{\gamma b(x) \cos\theta}d\theta.
	\end{align*}
	Note that $fe^{-\gamma b(\cdot)e^{i\theta}}\in L^p(w e^{pb\gamma\cos\theta})$
	and $\|fe^{-\gamma b(\cdot)e^{i\theta}}\|_{L^p(w e^{pb\gamma\cos\theta})}=\|f\|_{L^p(w)}$, using Lemma \ref{upperbound} and Minkowski's inequality again, it yields that
	\begin{align*}
		\|g_{\mu, \Psi,b}^*f\|_{L^p(w)}&\leq \frac{1}{2\pi\gamma}\Big(\int_{\mathbb R^n}\Big|\int_0^{2\pi}g_{\mu, \Psi}^* (fe^{-\gamma e^{i\theta}b(\cdot)})(x)e^{\gamma b(x) \cos\theta}d\theta\Big|^pw(x)dx\Big)^{\frac{1}{p}}\\
		&\leq\frac{1}{2\pi\gamma}\int_0^{2\pi}\Big(\int_{\mathbb R^n}g_{\mu, \Psi}^*(fe^{-b(\cdot)\gamma e^{i\theta}})^p(x)e^{pb(x)\gamma \cos\theta}w(x)dx\Big)^{\frac{1}{p}} d\theta\\
		&\leq C[w]_{A_p}^{\alpha_p+1/(p-1)}
		\frac{1}{2\pi\gamma}\int_0^{2\pi}
		\|fe^{-\gamma b(\cdot)e^{i\theta}}\|_{L^p(w e^{pb\gamma\cos\theta})}d\theta\\
		&=C[w]_{A_p}^{\alpha_p+1/(p-1)}\|f\|_{L^p(w)},
		%&\leq C_{n,p,b,\epsilon}[w]_{A_p}^{\max\{1,\frac{1}{p-1}\}}\|f\|_{L^p(w)},
	\end{align*}
	which completes the proof of Thorem \ref{lebesgue bound}.
\end{proof}

\medspace
To prove Theorem \ref{morreybound}, we need the following estimate for sharp maximal function of $g_{\mu, \Psi}^*$.
\begin{lemma}\label{sharpbound}
	Let $L$ be a non-negative self-adjoint operator, such that the corresponding heat kernels satisfy Gaussian bounds \eqref{(GE)}. If $\mu>3$ and $0 < \delta <\varepsilon < 1$, then for any $b\in \|b\|_{Osc_{\exp L^r}},\ r\geq1$, there is a constant $C>0$ such that
	\begin{equation}\label{sharp}
		M_\delta^{\sharp}\left(g_{\mu, \Psi,b}^*f\right)(x) \leq C \|b\|_{Osc_{\exp L^r}}\Big(M_{\varepsilon}(g_{\mu, \Psi}^*({f}))(x)+M_{L(\log L)^{1/r}} f(x) \Big).
	\end{equation}
\end{lemma}

\proof 
Let $\mu>3$ and $\eta>1$. To prove \eqref{sharp}, we will show that for each ball $B$ containing $x$ and for some constant $c_B$, there exists a positive constant $C$, such that
\begin{equation}\label{key}
	\Big(\frac{1}{|B|} \int_B\big|g_{\mu, \Psi,b}^*(f)(z)^\delta-|c_B|^\delta\big| d z\Big)^{\frac{1}{\delta}} \leq CM_{L(\log L)^{1/r}} f(x)+CM_{\varepsilon_0}(g_{\mu, \Psi}^*({f}))(x) .
\end{equation}
For $0<\delta<\frac{1}{2}$,
\begin{align*}
	\Big(\frac{1}{|B|} \int_B\big|g_{\mu, \Psi,b}^*(f)(z)^\delta-|c_B|^\delta\big| dz\Big)^{\frac{1}{\delta}} 
	\leq &\Big(\frac{1}{|B|} \int_B\big|g_{\mu, \Psi,b}^*(f)(z)-c_B|^\delta\big| dz\Big)^{\frac{1}{\delta}}\\
	\leq & C_\delta\Big(\frac{1}{|B|} \int_B|b(z)-\langle b\rangle_{B^*}|g_{\mu, \Psi}^*(f)(z)^\delta dz\Big)^{\frac{1}{\delta}}\\
	&+ C_\delta\Big(\frac{1}{|B|} \int_B\big|g_{\mu, \Psi}^*((b-\langle b\rangle_{B^*})f)(z)-c_B\big|^\delta dz\Big)^{\frac{1}{\delta}}\\
	=&:I+II,
\end{align*}
where $B^*=2B$ and $\langle b\rangle_{B^*}:=\frac{1}{|B^*|} \int_{B^*}f(x)dx$.

Consider the first term $I$. Applying H\"older's inequality, it yields
\begin{align*}
	{I} &=C_\delta \Big(\frac{1}{|B|}\int_{B}|b-\langle b\rangle_{B^*}|g_{\mu, \Psi}^*({f})(z)^\delta dz\Big)^\frac{1}{\delta}\\
	&\leq \Big(\frac{1}{|B|}\int_B|b(z)-\langle b\rangle_{B^*}|^{\delta q}dz\Big)^\frac{1}{\delta q}\Big(\frac{1}{|B|}\int_B|g_{\mu, \Psi}^*({f})|^{\delta q'}dz\Big)^\frac{1}{\delta q'}\\
	&\leq C_{\delta ,n,q}\|b\|_{\BMO}M_{\delta q'}(g_{\mu, \Psi}^*({f}))(x)\\
	&\leq C_{\delta ,n,q,r}\|b\|_{Osc_{\exp L^r}}M_{\varepsilon_0}(g_{\mu, \Psi}^*({f}))(x),
\end{align*}
where the last inequality is given by
$\|b\|_{\BMO}\leq \|b\|_{Osc_{\exp L^r}}$ and $\varepsilon_0\geq \delta q'$.

Let $T(B)=\left\{(y, t): y \in B, 0<t<r_B\right\}$, where $r_B$ denotes the radius of $B$. For $(y, t) \in T(B)$, using \eqref{kerelsupp} of Lemma \ref{kernel}, we have
\begin{equation}\label{support}
	\Psi(t \sqrt{L}) f(y)=\Psi(t \sqrt{L})\left(f \chi_{3 B}\right)(y) .
\end{equation}
Now, fix a ball $B$ containing $x$. Denote $\mathbb{R}_{+}^n=\mathbb{R}^n \times(0, \infty)$. For any $z \in B$, we decompose $\left(g_{\mu, \Psi,b}^*(f)(z)\right)^2$ into the sum of	
$$
A_1(z)^2=\iint_{T(2 B)}|\Psi(t \sqrt{L}) ((b-\langle b\rangle_{B^*})f)(y)|^2\left(\frac{t}{t+|z-y|}\right)^{n \mu} \frac{d y d t}{t^{n+1}},
$$
and
$$
A_2(z)^2=\iint_{\mathbb{R}_{+}^n \backslash T(2 B)}|\Psi(t \sqrt{L}) ((b-\langle b\rangle_{B^*})f)(y)|^2\left(\frac{t}{t+|z-y|}\right)^{n \mu} \frac{d y d t}{t^{n+1}} .
$$
Then, since $\left.|| a\right|^s-|b|^s|\leq| a-\left.b\right|^s$ for $0<s<1$, we have
$$
\begin{aligned}
	\left|g_{\mu, \Psi}^*((b-\langle b\rangle_{B^*})f)(z)-c_B\right| & =\left|\left(A_1(z)^2+A_2(z)^2\right)^{1 / 2}-c_B\right| \\
	& \leq\left|A_1(z)^2+A_2(z)^2-c_B^2\right|^{1 / 2} \\
	& \leq A_1(z)+\left|A_2(z)^2-c_B^2\right|^{1 / 2} .
\end{aligned}
$$
Now,  we have
$$
II \leq C_\delta\Big(\frac{1}{|B|} \int_B A_1(z)^{\delta} dz\Big)^{\frac{1}{\delta}}+ C_\delta\Big(\frac{1}{|B|} \int_B\big|A_2(z)^2-c_B^2\big|^{\frac{\delta}{2}} dz\Big)^{\frac{1}{\delta}}.
$$
{Note that $ g_{\mu, \Psi}^*(f)$ maps  $L^1\left(\mathbb{R}^n\right)$ into $L^{1,\infty}\left(\mathbb{R}^n\right)$ due to Theorem 4.19 in \cite{uhl}. }Thus, by Kolmogorov's inequality, it yields that
\begin{equation}\label{II1}
	\begin{aligned}
		\frac{1}{|B|} \int_B A_1(z) dz& =\frac{1}{|B|}\int_0^{\infty}\left|\left\{z \in B: g_{\mu, \Psi}^*\left((b-\langle b\rangle_{B^*})f \chi_{6 B}\right)(z)>t\right\}\right| d t \\
		& \leq \frac{C}{|B|}\int_{6 B}|b(z)-\langle b\rangle_{B^*}||f(z)|dz\\
		& \leq C\|b\|_{Osc_{\exp L^r}} M_{L(\log L)^{1/r}} f(x).
	\end{aligned}
\end{equation}
We take $c_B=A_2\left(z_B\right)^{1 / 2}$, where $z_B$ is the center of $B$. 
By the mean value theorem, we know that for $z \in B$ and $(y, t) \notin T(2 B)$, there exists $0<s \leq 1$, such that
$$
(t+|z-y|)^{-n \mu}-\left(t+\left|z_B-y\right|\right)^{-n \mu} \leq {Cr}_B^s(t+|z-y|)^{-n \mu-s} .
$$
From this and \eqref{support}, using Lemma \ref{kernel}, H\"older's inequality and $\mu>3$, we get
\begin{align*}
	&\big|A_2(z)^2-c_B^2\big|\\
	& \leq C r_B^s \iint_{\mathbb{R}_{+}^n \backslash T(2 B)} t^{n \mu}|\Psi(t \sqrt{L}) ((b-\langle b\rangle_{B^*})f)(y)|^2\left(\frac{1}{(t+|z-y|)}\right)^{n \mu+s} \frac{d y d t}{t^{n+1}} \\
	& \leq C \sum_{k=1}^{\infty} \frac{1}{2^{s k}\left(2^k r_B\right)^{n \mu}} \iint_{T\left(2^{k+1} B\right) \backslash T\left(2^k B\right)}|\Psi(t \sqrt{L}) ((b-\langle b\rangle_{B^*})f)(y)|^2 \frac{d y d t}{t^{n+1-n \mu}} \\
	& \leq C \sum_{k=1}^{\infty} \frac{1}{2^{s k}\left(2^k r_B\right)^{n \mu}}\left(\int_0^{2^{k+1} r_B} \int_{2^{k+1} B} \frac{d y d t}{t^{1+3 n-n \mu}}\right)\left(\int_{6 \cdot 2^k B}|f(y)||b(y)-\langle b\rangle_{B^*}| d y\right)^2 \\
	& \leq C \sum_{k=1}^{\infty} \frac{1}{2^{s k}}\left(\frac{1}{\left|2^{k+1} B\right|} \int_{6 \cdot 2^k B}|b(y)-\langle b\rangle_{B^*}||f(y)| d y\right)^2\\
	& \leq C\sum_{k=1}^{\infty} \frac{1}{2^{s k}} \|b\|_{Osc_{\exp L^r}}^2 M_{L(\log L)^{1/r}} f(x)^2
\end{align*}
for all $z \in B$. Combining this estimate with \eqref{II1} yields
$$
II\leq C \|b\|_{Osc_{\exp L^r}} M_{L(\log L)^{1/r}} f(x),
$$
and then the desired estimate \eqref{key} holds. This concludes the proof of this lemma.
\qed

\medskip
Now we give the proof of Theorem \ref{morreybound}.

\medskip
\textbf{The proof of Theorem \ref{morreybound}}. Let $\mu>3,1<p<\infty$, and $w \in A_p$. Then, there exists $1<\varepsilon<p$ such that $w \in A_{p /\varepsilon}$. Using Lemmas \ref{lemwi} and \ref{sharpbound}, we obtain that for $0 < \delta <\varepsilon < 1$,
$$
\begin{aligned}
	\left\|g_{\mu, \Psi,b}^*(f)\right\|_{L^{p, \kappa}(w)}
	& \leq \|M_\delta\left(g_{\mu, \Psi,b}^*f\right)\|_{L^{p, \kappa}(w)} \\
	& \leq C\big\|M_\delta^{\sharp}\left(g_{\mu, \Psi,b}^*f\right)\big\|_{L^{p, \kappa}(w)} \\
	&\leq C \|b\|_{Osc_{\exp L^r}}\Big(\|M_{L(\log L)^{1/r}} f\|_{L^{p, \kappa}(w)} +\|M_{\varepsilon}(g_{\mu, \Psi}^*{f})\|_{L^{p, \kappa}(w)}\Big).\\
\end{aligned}
$$
Lemma \ref{maximal}, combining with Lemma \ref{th3.2}, gives that
$$
\begin{aligned}
	\|M_{\varepsilon}(g_{\mu, \Psi}^*{f})\|_{L^{p, \kappa}(w)}&=C\left\|M\left(|g_{\mu, \Psi}^*{f}|^\varepsilon\right)\right\|_{L^{p /\varepsilon, \kappa}(w)}^{1 /\varepsilon} \\
	& \leq C\left\||g_{\mu, \Psi}^*{f}|^\varepsilon\right\|_{L^{p / \varepsilon, \kappa}(w)}^{1 / \varepsilon}=C\|g_{\mu, \Psi}^*{f}\|_{L^{p, \kappa}(w)}\leq C\|{f}\|_{L^{p, \kappa}(w)}.
\end{aligned}
$$
On the other hand,  Since $r\geq 1$, it follows that $M_{L(\log L)^{1/r}}$ is pointwise 
smaller than $M_{L(\log L)}$ which is known to be equivalent to $M^2$. Therefore
$$
\|M_{L(\log L)^{1/r}} f\|_{L^{p, \kappa}(w)}\leq C\|f\|_{L^{p, \kappa}(w)}.
$$
The proof of Theorem \ref{morreybound} is completed.
\qed

\medskip
By mimicking the proof of Theorem 1.5 and 1.6 in \cite{[PT]}, using Lemma \ref{lemwi} and Lemma \ref{sharpbound}, it's easy to obtain the following weighted endpoint estimates of $g_{\mu, \Psi,b}^*$. We omit the proof.
\begin{theorem}\label{wgbF-S}
	Let $\Phi(t)=t(1+\log^+t)^r,\ \omega\in A_\infty$
	and	$b\in Osc_{\exp L^r}, r\geq 1$.Then there exists a constant $C$ that depends on the doubling condition of $\Phi$ and the $A_\infty$ condition of $\omega$ such that for any bounded function ${f}$ with a compact supported set, we have
	\begin{equation}\label{w1}
		\begin{aligned}
			&\sup_{\la>0}\frac{1}{\Phi(\frac{1}{\la})}\omega(\{y\in {\rn}:|g_{\mu, \Psi,b}^*({f})(y)|>\la\})\\
			&\leq C\sup_{\la>0}\frac{1}{\Phi(\frac{1}{\la})}\omega(\{y\in\rn:\mathcal{M}_\Phi(\|b\|_{Osc_{\exp L^r}}{f})(y)>\la\}).
		\end{aligned}
	\end{equation}
	Also, \eqref{w1} holds for $S_{P,b}$ and $S_{H,b}$.	
\end{theorem}
In particular, if $w\in A_1$, we have the following weighted weak estimate. 
\begin{corollary}\label{wgbend}
	Let $\Phi(t)=t(1+\log^+t)^r,\ \omega\in A_{{1}}$ and
	$b\in Osc_{\exp L^r},\ r\geq 1$. Then there exists a constant $C$ that depends on the doubling condition of $\Phi$ and the $A_{{1}}$ condition of ${\omega}$ such that for any bounded function ${f}$ with a compact supported set, we have
	\begin{equation}\label{w2}
		\begin{aligned}
			{{\omega}}(\{y\in {\rn}:|g_{\mu, \Psi,b}^*({f})(y)|>\la\})
			\leq C\int_{\rn}\Phi\Big(\frac{\|b\|_{Osc_{\exp L^r}}|f(x)|}{\la}\Big)\omega(x)dx.	
		\end{aligned}
	\end{equation}
	Also, \eqref{w2} holds for $S_{P,b}$ and $S_{H,b}$.	
\end{corollary}

Finally, we give the proof of Theorem \ref{compact thm}. To do this, we shall give the Frechet-Kolmogorov
theorem on weighted Morrey spaces, which was given by Liu and Cui in \cite{LC} and plays a key role in the proof of Theorem \ref{compact thm}.

\medskip
\begin{lemma}[\cite{LC}]\label{compactcondition}
	Let $1 \leqslant p<\infty, 0<\lambda<n$. If the subset $G$ in $L^{p, \lambda}\left(w\right)$ satisfies the following conditions:
	
	\begin{itemize}
		\item[\rm(a)]  $\mathop{\sup}\limits _{f \in G}\|f\|_{L^{p, \lambda}(w)}<\infty$;
		\item[\rm(b)] 	$\mathop{\lim}\limits _{|y| \rightarrow 0}\|f(\cdot+y)-f(\cdot)\|_{L^{p, \lambda}(w)}=0, \quad \text { uniformly in } f \in G$;
		\item[\rm(c)] $\mathop{\lim}\limits _{\beta \rightarrow \infty}\left\|f \chi_{E_\beta}\right\|_{L^{p, \lambda}(w)}=0, \quad \text { uniformly in } f \in G$,
	\end{itemize}
	where $E_\beta=\left\{x \in \mathbb{R}^n:|x|>\beta\right\}$, then $G$ is a strongly pre-compact set in $L^{p, \lambda}(w)$.
\end{lemma}

\medskip
\noindent\textbf{Proof of Theorem \ref{compact thm}.}
We first prove the result for $S_{H,b}$.  According to the assumption \eqref{smooth}, the kernel of $t^2 L e^{-t^2 L}$, $q_t$, satisfies that for any $\alpha \in(0, \nu)\  (0<\nu\leq 1)$, there exists a positive constant $C_1$ such that for all $t \in(0, \infty)$ and almost every $x, y, h \in \mathbb{R}^n$ with $2|h| \leq t+|x-y|$,
\begin{equation}\label{smooth2}
	\begin{aligned}
		& \left|q_{t}(x+h, y)-q_{t}(x, y)\right|+\left|q_{t}(x, y+h)-q_{t}(x, y)\right| \\
		& \quad \leq \frac{C}{t^n}\left(\frac{|h|}{t+|x-y|}\right)^\alpha \exp \left\{-\frac{c|x-y|^{2}}{t^{2}}\right\},
	\end{aligned}
\end{equation}
\begin{equation}\label{size2}
	\left|q_t(x, y)\right|+\left|q_t(y, x)\right| \leq \frac{C_1}{t^{n}} \exp \Big\{-\frac{c_1|x-y|^{2}}{t^2}\Big\}.
\end{equation}
See \cite{BCD} for details. 

Suppose that $F$ is an arbitrary bounded set in $L^{p, \lambda}\left(w\right)$, that is, there exists a constant $D>0$ such that $\|f\|_{L^{p, \lambda}(w)} \leqslant D$ for every $f \in F$. By Corollary \ref{morreybound-bmo}, it holds that
\begin{equation}\label{mb}
	\left\|[b, S_H] f\right\|_{L^{p, \lambda}(w)} \leqslant C\|b\|_{\BMO}\|f\|_{L^{p, \lambda}(w)}.
\end{equation}
Thus $[b, S_H]$ is continuous in $L^{p, \lambda}\left(w\right)$. Let $G=\left\{[b, S_H] f: f \in F\right\}$ if $b \in C_0^{\infty}\left(\mathbb{R}^n\right)$ and $\widetilde{G}=\left\{[b, S_H] f: f \in F\right\}$ if $b \in \operatorname{CMO}\left(\mathbb{R}^n\right)$. So, by the definition of compactness, it suffices to prove that for any bounded set $F$ in $L^{p, \lambda}\left(w\right), \widetilde{G}$ is a strongly pre-compact set in $L^{p, \lambda}\left(w\right)$. We first show that if (a)-(c) of Lemma \ref{compactcondition} hold uniformly in $G$, then (a)-(c) also hold  uniformly in $\widetilde{G}$ and thus $[b,  S_H]$ is a compact operator in $L^{p, \lambda}\left(w\right)$.

Actually, assume that $b \in \operatorname{CMO}\left(\mathbb{R}^n\right)$, then for any $\varepsilon>0$ there exists $b^{\varepsilon} \in C_0^{\infty}\left(\mathbb{R}^n\right)$ such that $\left\|b-b^{\varepsilon}\right\|_{\BMO}<\varepsilon$. By the inequality
$$
\begin{aligned}
	& \left|\left[b, S_H\right] f(x)-\left[b^{\varepsilon},  S_H\right] f(x)\right| \\
	&\leqslant \left(\int_0^{\infty} \int_{|x-y|<t}\left|\int_{\rn} q_t(y,z)\left(\left(b-b^{\varepsilon}\right)(x)-\left(b-b^{\varepsilon}\right)(z)\right) f(z) d z\right|^2 \frac{d y d t}{t^{n+1}}\right)^{\frac{1}{2}}
\end{aligned}
$$
and \eqref{mb}, we get
$$
\left\|\left[b,  S_H\right]-\left[b^{\varepsilon},  S_H\right]\right\|_{L^{p, \lambda} (w)\rightarrow L^{p, \lambda}(w)} \leqslant\left\|\left[b-b^{\varepsilon},  S_H\right]\right\|_{L^{p, \lambda}(w) \rightarrow L^{p, \lambda}(w)} \leqslant C \varepsilon.
$$
Thus, by Minkowski's inequality, for any fixed $f \in F$,
$$
\sup _{f \in F}\left\|\left[b, S_H\right] f\right\|_{L^{p, \lambda}(w)} \leqslant \sup _{f \in F}\left\|\left[b^{\varepsilon},  S_H\right] f\right\|_{L^{p, \lambda}(w)}+C D \varepsilon<\infty.
$$
Similarly, by Minkowski's inequality, for any fixed $f \in F$
$$
\begin{aligned}
	\lim _{\beta \rightarrow \infty}\left\|\left[b,  S_H\right] f \chi_{E_\beta}\right\|_{L^{p, \lambda}(w)}& \leqslant \lim _{\beta \rightarrow \infty}\left\|\left[b^{\varepsilon},  S_H\right] f \chi_{E_\beta}\right\|_{L^{p, \lambda}(w)}+\lim _{\beta \rightarrow \infty}\left\|\left[b-b^{\varepsilon}, S_H\right] f \chi_{E_\beta}\right\|_{L^{p, \lambda}(w)} \\
	& \leqslant C D \varepsilon.
\end{aligned}
$$
Thus (a)-(c) hold uniformly for $\tilde{G}$. Therefore, by Lemma \ref{compact thm}, we know $\tilde{G}$ is a strongly pre-compact set in $L^{p, \lambda}\left(w\right)$ and then $\left[b,  S_H\right]$ is a compact operator in $L^{p, \lambda}\left(w\right)$.
So, it suffices to prove that (a)-(c) of Lemma \ref{compactcondition} hold uniformly in $G$. Without loss of the generality, we can assume $\|b\|_{\BMO}=1$. By \eqref{mb}, we have
$$
\sup _{f \in F}\left\|\left[b,  S_H\right] f\right\|_{L^{p, \lambda}(w)} \leqslant C\|b\|_{\BMO}\|f\|_{L^{p, \lambda}(w)} \leqslant C D<\infty,
$$
which means (a) holds for the commutator $[b,  S_H]$ in $G$ uniformly. Now we discuss (c).
\iffalse
For any $\varepsilon>0, q>1$ there exists $A>0$ such that
$$
\left(\int_A^{\infty} \frac{d r}{r^{n q-n+1}}\right)^{\frac{1}{q}}<\varepsilon.
$$ \fi
Suppose that supp $b \subset\{z:|z|<R\}$ for some $R>0$ and $\beta>2 R$. Then for any $x$ satisfying $|x|>\beta$ and $z\in \supp b$,
we have $b(x)=0$.
Thus for $|x|>\beta$, by Minkowski's inequality,  we have
\begin{align*}
	\left|\left[b,  S_H\right] f(x)\right| 
	& = \left(\int_0^{\infty} \int_{|x-y|<t}\Big|\int_{|z|<R} q_t(y,z)(b(x)-b(z)) f(z) d z\Big|^2 \frac{d y d t}{t^{n+1}}\right)^{1 / 2} \\
	&= \left(\int_0^{\infty} \int_{|x-y|<t}\Big|\int_{|z|<R}q_t(y,z)b(z) f(z) d z\Big|^2 \frac{d y d t}{t^{n+1}}\right)^{1 / 2}\\
	&\leqslant C \int_{|z|<R}|b(z)||f(z)|\left(\int_0^{|x-z| / 2}\int_{|x-y|<t}  |q_t(y,z)|^2
	\frac{d y d t}{t^{n+1}}\right)^{1 / 2} d z \\
	&\quad+C \int_{|z|<R}|b(z)||f(z)|\left(\int_{|x-z| / 2}^{\infty} \int_{|x-y|<t}  |q_t(y,z)|^2\frac{d y d t}{t^{n+1}}\right)^{1 / 2} d z\\
	&=:J_1+J_2.
\end{align*}
For $J_1$, we have $|y-z|\geq|x-z|-|x-y|\geq|x-z|/2$ due to $|x-y|<|x-z|/2$. Then using \eqref{size2} and the estimate $e^{-s}\leq\frac{C}{s^{M/2}}$ with $M>2n$, it yields that

\begin{equation}\label{j1}
	\begin{aligned}
		J_1 &\leq C \int_{|z|<R}|b(z)||f(z)|\left(\int_0^{|x-z| / 2}\int_{|x-y|<t} \exp\Big\{-\frac{c|y-z|^2}{t^2}\Big\}
		\frac{d y d t}{t^{3n+1}}\right)^{1 / 2} d z \\
		&\leq C \int_{|z|<R}|b(z)||f(z)|\left(\int_0^{|x-z| / 2}\int_{|x-y|<t} \frac{t^M}{|x-z|^M}
		\frac{d y d t}{t^{3n+1}}\right)^{1 / 2} d z \\
		&\leqslant  C\|b\|_\infty \int_{|z|<R}|x-z|^{-n}|f(z)| d z .
	\end{aligned}
\end{equation}
For $J_2$, 
\begin{equation}\label{j2}
	\begin{aligned}
		J_2 &\leq C \int_{|z|<R}|b(z)||f(z)|\left(\int_{|x-z| / 2}^{\infty} \int_{|x-y|<t}  
		\frac{d y d t}{t^{3n+1}}\right)^{1 / 2} d z \\
		&\leqslant  C\|b\|_\infty \int_{|z|<R}|x-z|^{-n}|f(z)| d z. \\
	\end{aligned}
\end{equation}
According to \eqref{j1} and \eqref{j2}, since $|x|\leq 2|x-z|$ for $|x|>\beta>2R$ and $z\in \supp b$, we can get
\begin{equation}
	\begin{aligned}
		\left|\left[b,  S_H\right] f(x)\right|	&\leqslant  C\|b\|_\infty \int_{|z|<R}|x-z|^{-n}|f(z)| d z \\
		&	\leqslant  C\|b\|_\infty\int_{|z|<R}|x|^{-n}|f(z)| dz	.
	\end{aligned}
\end{equation}
Denote $B_R=\{z\in\rn:|z|<R\}$, one can obtain via H\"older's inequality that
\begin{equation}
	\begin{aligned}
		\int_{B_R}|f(z)| dz
		&\leq \Big(\int_{B_R}|f(y)|^pw(z)dz\Big)^{\frac{1}{p}}\Big(\int_{B_R}w(z)^{-\frac{p'}{p}}dz\Big)^{\frac{1}{p'}}\\
		&\leq [w]_{A_p}^{\frac{1}{p}}\|f\|_{L^{p,\lambda}(w)}|B_R|w(B_R)^{(\frac{\lambda}{n}-1)\frac{1}{p}}=:\mathcal{C}(p,R,\lambda).
	\end{aligned}
\end{equation}	
By the property of $A_p$ weights in Lemma \ref{weights}, we can choose a $1<q<p<\infty$ so that $w\in A_q.$ Hence, we have for any $\tilde{B}\subset\rn$,

\begin{align*}
	& \frac{1}{w(\tilde{B})^{{\lambda}/{n}}} \int_{\tilde{B}}\left|\left[b,  S_H\right] f(x)\right|^pw(x) \chi_{\{y:|y|>\beta\}}(x) d x \\	
	& \leq C\Big(\int_{|z|<R}|f(z)| dz\Big)^{p}\frac{1}{w(\tilde{B})^{{\lambda}/{n}}} \int_{\tilde{B}} \frac{w(x)}{|x|^{np}}\chi_{\{y:|y|>\beta\}}(x) d x \\	
	& \leq C\mathcal{C}(p,R,\lambda)^p	\frac{1}{w(\tilde{B})^{{\lambda}/{n}}} \int_{\tilde{B}\cap\{y:|y|>\beta\}} \frac{w(x)}{|x|^{np}} d x \\	
	& \leq C\mathcal{C}(p,R,\lambda)^p	\frac{1}{w(\tilde{B})^{{\lambda}/{n}}} \sum_{j=0}^\infty\int_{\tilde{B}\cap(B(0,2^{j+1}\beta)
		\backslash B(0,2^{j}\beta))} \frac{w(x)}{|x|^{np}} d x \\	
	& \leq C\mathcal{C}(p,R,\lambda)^p \sum_{j=0}^\infty(2^j\beta)^{-np} w(\tilde{B}\cap(B(0,2^{j+1}\beta)
	\backslash B(0,2^{j}\beta)))^{1-{\lambda}/{n}}\\
	& \leq C\mathcal{C}(p,R,\lambda)^p \sum_{j=0}^\infty(2^j\beta)^{-np} w(B(0,2^{j+1}\beta)
	)^{1-{\lambda}/{n}}\\	
	& \leq C[w]_{A_q}\mathcal{C}(p,R,\lambda)^p \sum_{j=0}^\infty(2^j\beta)^{-np}(2^{j+1}\beta)^{q(n-\lambda)} w(B(0,1)
	)^{1-{\lambda}/{n}}\\
	& \leq C[w]_{A_q}\mathcal{C}(p,R,\lambda)^p \beta^{n(q-p)-\lambda q}.
\end{align*}
Therefore for any $f\in F$, by using Lemma \ref{weights}, it leads to
$$
\begin{aligned}
	& \Big(\frac{1}{w(\tilde{B})^{{\lambda}/{n}}} \int_{\tilde{B}}\left|\left[b,  S_H\right] f(x)\right|^pw(x) \chi_{\{y:|y|>\beta\}}(x) d x \Big)^{1/p}\\
	& \leq C|B_R|w(B_R)^{(\frac{\lambda}{n}-1)\frac{1}{p}} \beta^{n(q-p)-\lambda q}\|f\|_{L^{p,\lambda}(w)}\\
	& \leq C[w]_{A_p}R^{{\lambda}-n} R^n\beta^{n(q-p)-\lambda q}\|f\|_{L^{p,\lambda}(w)}\\
	& \leq C \beta^{n(q-p)-\lambda q+\lambda}\|f\|_{L^{p,\lambda}(w)}\rightarrow 0, \ \hbox{as\ }\beta\rightarrow \infty,
\end{aligned}
$$
which shows that (c) of Lemma \ref{compactcondition} holds for the commutator $\left[b, S_H\right]$ in $G$ uniformly.

Finally, it remains to show (b) of Lemma \ref{compactcondition} holds for the the commutator $\left[b, S_H\right]$ in $G$ uniformly. We need to prove that for any fixed $\varepsilon>0$, if $|h|$ is sufficiently small depended only on $\varepsilon$, then for every $f \in F$,
\begin{equation}\label{3.14}
	\left\|\left[b, S_H\right] f(\cdot+h)-\left[b, S_H\right] f(\cdot)\right\|_{L^{p, \lambda}(w)} \leqslant C \varepsilon.
\end{equation}
To do this, we need the smoothness truncated technique.
Let $\varphi\in C^{\infty}([0,\infty))$ satisfy
\begin{eqnarray}\label{3.1}
	0\leq\varphi\leq1\ \ \ \text{and}\ \ \
	\varphi(t)=
	\begin{cases}
		1,          &t\in[0,1],\\
		0,          &t\in[2,\infty).
	\end{cases}
\end{eqnarray}
For any $\gamma>0$, let
\begin{equation}\label{3.2}
	q_{t,\gamma}(x,y)=q_t(x,y)\Big(1-\varphi\Big(\frac{t}{\gamma}\Big)\Big).
\end{equation}
Note that this truncated function essentially just truncates $t$ and does not contribute anything to $x$ and $y$. With the kernel $q_{t,\gamma}$, we define
$$
\begin{aligned}
	& S_{H,\gamma} f(x)=\left(\int_0^{\infty} \int_{|x-y|<t}\left|\int_{\rn} q_{t,\gamma}(y,z) f(z) d z\right|^2 \frac{d y d t}{t^{n+1}}\right)^{\frac{1}{2}}
\end{aligned}
$$
and
$$
\begin{aligned}
	\left[b, S_{H,\gamma}\right] f(x) =\left(\int_0^{\infty} \int_{|x-y|<t}\left|\int_{\rn} q_{t,\gamma}(y,z)\left(b(x)-b(z)\right) f(z) d z\right|^2 \frac{d y d t}{t^{n+1}}\right)^{\frac{1}{2}}.
\end{aligned}
$$
Then,  by \eqref{size2} and  \eqref{smooth2}, for any $\alpha \in(0, \nu)$, there exists a positive constant $C_1$ such that for all $t \in(0, \infty)$ and almost every $x, y, h \in \mathbb{R}^n$ with $2|h| \leq t+|x-y|$, we have
\begin{align}\label{3.5'}
	\begin{split}
		|q_{t,\gamma}(x,y)|\leq |q_t(x,y)|
		\leq\frac{C}{t^{n}} \exp \Big\{-\frac{c|x-y|^{2}}{t^2}\Big\},
	\end{split}
\end{align}
and
\begin{align}\label{3.5}
	\begin{split}
		|q_{t,\gamma}(x+h,y)-q_{t,\gamma}(x,y)|
		&\leq |q_{t}(x+h,y)-q_{t}(x,y)|\\
		&\leq \frac{C_1}{t^n}\left(\frac{|h|}{t+|x-y|}\right)^\alpha \exp \left\{-\frac{c_1|x-y|^{2}}{t^{2}}\right\}.
	\end{split}
\end{align}
For any  $\gamma>0$, by \eqref{3.2}, one has
\begin{align}\label{3.6}
	&|\left[b, S_{H,\gamma}\right]f(x)-\left[b, S_{H}\right] f(x)|\\
	&\leq C\displaystyle\Big(\int_{0}^{2\gamma}\int_{|x-y|<t}\Big(\int_{\mathbb{R}^n}t^{-n}e^{-\frac{c|y-z|^2}{t^2}}|b(x)-b(z)||f(z)|dz\Big)^2 \frac{d y d t}{t^{n+1}}\Big)^{1/2}\nonumber\\
	&\leq C\displaystyle\Big(\int_{0}^{2\gamma}\int_{|x-y|<t}\Big(\int_{|x-y|\geq\frac{|x-z|}{2}}t^{-n}e^{-\frac{c|y-z|^2}{t^2}}|b(x)-b(z)||f(z)|dz\Big)^2 \frac{d y d t}{t^{n+1}}\Big)^{1/2}\nonumber\\
	&\quad +C\displaystyle\Big(\int_{0}^{2\gamma}\int_{|x-y|<t}\Big(\int_{|x-y|<\frac{|x-z|}{2}}t^{-n}e^{-\frac{c|y-z|^2}{t^2}}|b(x)-b(z)||f(z)|dz\Big)^2 \frac{d y d t}{t^{n+1}}\Big)^{1/2}\nonumber\\
	&=:U_1+U_2.  \nonumber
\end{align}
For $U_1$, note that $|x-z|\leq{2}|x-y|\leq 2t$, we have
\begin{align}\label{3.7}
	\begin{split}
		U_1&\leq C\|\nabla b\|_{L^\infty}\displaystyle\Big(\int_{0}^{2\gamma}\int_{|x-y|<t}\Big(\int_{|x-z|\leq 2t}t^{-n}|x-z||f(z)|dz\Big)^2 \frac{d y d t}{t^{n+1}}\Big)^{1/2}\\
		&\leq C\|\nabla b\|_{L^\infty}\displaystyle\Big(\int_{0}^{2\gamma}t^2Mf(x)^2 \frac{d y d t}{t}\Big)^{1/2}\\
		&\leq C\gamma Mf(x).
	\end{split}
\end{align}
For $U_2$, it not difficult to see that $|y-z|\geq \frac{|x-z|}{2}$, then using the estimate $e^{-s}\leq\frac{C}{s^{M/2}}$ with $M>n+1$ and splitting to annuli, it follows that
\begin{align}\label{3.8}
	U_2&\leq C\displaystyle\Big(\int_{0}^{2\gamma}\int_{|x-y|<t}\Big(\int_{2|x-y|<|x-z|<t}t^{-n}e^{-\frac{c|y-z|^2}{t^2}}|b(x)-b(z)||f(z)|dz\Big)^2 \frac{d y d t}{t^{n+1}}\Big)^{1/2}\nonumber\\
	&\quad+C\displaystyle\Big(\int_{0}^{2\gamma}\int_{|x-y|<t}\Big(\int_{|x-y|<\frac{|x-z|}{2},|x-z|\geq t}t^{-n}e^{-\frac{c|y-z|^2}{t^2}}|b(x)-b(z)||f(z)|dz\Big)^2 \frac{d y d t}{t^{n+1}}\Big)^{1/2}\nonumber\\
	&\leq C\gamma Mf(x)+C\|\nabla b\|_{L^\infty}\displaystyle\Big(\int_{0}^{2\gamma}\int_{|x-y|<t}\Big(\sum_{k=0}^{\infty}\int_{2^kt\leq|x-z|\leq 2^{k+1} t}t^{M-n}\frac{|x-z|}{|x-z|^M}|f(z)|dz\Big)^2 \frac{d y d t}{t^{n+1}}\Big)^{1/2}\nonumber\\
	&\leq C\gamma Mf(x)+C\displaystyle\Big(\int_{0}^{2\gamma}\Big(\sum_{k=0}^{\infty}\frac{t^{M-n}}{(2^kt)^{M-n-1}}\frac{1}{(2^kt)^n}\int_{|x-z|\leq 2^{k+1} t}|f(z)|dz\Big)^2 \frac{ d t}{t}\Big)^{1/2}\nonumber\\
	&\leq C\gamma Mf(x)+C\sum_{k=0}^{\infty}\frac{1}{2^{k(M-n-1)}}\displaystyle\Big(\int_{0}^{2\gamma}t^2Mf(x)^2 \frac{ d t}{t}\Big)^{1/2}\nonumber\\
	&\leq C\gamma Mf(x).
\end{align}
Combing \eqref{3.8} with \eqref{3.6} and \eqref{3.7} may lead to
\begin{equation*}
	|\left[b, S_{H,\gamma}\right]f(x)-\left[b, S_{H}\right]f(x)|\leq C\gamma Mf(x).
\end{equation*}
Then, 
\begin{equation}\label{3.9}
	\|\left[b, S_{H,\gamma}\right]f-\left[b, S_{H}\right]f\|_{L^{p,\lambda}(w)}\leq C\gamma\|f\|_{L^{p,\lambda}(w)},
\end{equation}
which implies that
\begin{equation}\label{3.10}
	\lim_{\gamma\rightarrow 0}\|\left[b, S_{H,\gamma}\right]-\left[b, S_{H}\right]f\|_{{L^{p,\lambda}(w)}\rightarrow {L^{p,\lambda}(w)}}=0.
\end{equation}
Thus, by Minkowski's inequality, to prove \eqref{3.14}, it suffices to show that $\left[b, S_{H,\gamma}\right]$ satisfies \eqref{3.14}  when $\gamma>0$ is small enough.

Taking $\gamma=|h|$ and $|h|\in(0,1)$. If $t<|h|$, it holds that
$$\varphi\Big(\frac{t}{\gamma}\Big)=\varphi\Big(\frac{t}{|h|}\Big)=1.$$
This implies
$q_{t,\gamma}(x+h,y)=q_{t,\gamma}(x,y)=0.$
Therefore, we have
$$
\left|\left[b,S_{H,\gamma}\right] f(x+h)-\left[b,S_{H,\gamma}\right] f(x)\right| \leqslant\left(\int_{|h|}^{\infty} \int_{|x-y|<t}|I(x, h, y, t)|^2 \frac{d y d t}{t^{n+1}}\right)^{1 / 2},
$$
where
$$
\begin{aligned}
	I(x, h, y, t)= & \int_{\rn} q_{t,\gamma}(y,z)(b(x)-b(z)) f(z) d z  -\int_{\rn} q_{t,\gamma}(y+h,z)(b(x+h)-b(z)) f(z) d z.
\end{aligned}
$$
For any $0<\varepsilon<\frac{1}{3}$ and $h \in \mathbb{R}^n$, write $I(x, h, y, t)$ as
$$
\begin{aligned}
	I(x, h, y, t)=	&\int_{\rn} q_{t,\gamma}(y,z)(b(x)-b(x+h)) f(z) d z \\
	&+\int_{|x-z|<2^{\frac{1}{\varepsilon}}|h|}\left(q_{t,\gamma}(y+h,z)-q_{t,\gamma}(y,z)\right)(b(x+h)-b(z)) f(z) d z \\
	& +\int_{|x-z|\geq2^{\frac{1}{\varepsilon}}|h|}\left(q_{t,\gamma}(y+h,z)-q_{t,\gamma}(y,z)\right)(b(x+h)-b(z)) f(z) d z \\
	=&: \sum_{i=1}^3 I_i(x, h, y, t) \text {. }
\end{aligned}
$$
For $I_1(x, h, y, t)$, note that	
$$
S_1f(x):=\left(\int_{|h|}^{\infty} \int_{|x-y|<t}|I_1(x, h, y, t)|^2 \frac{d y d t}{t^{n+1}}\right)^{1 / 2}\leq\|\nabla b\|_{\infty}|h|S_H(f)(x),
$$
Applying \eqref{3.5} to $I_2$ and using Minkowsi's inequality again, we can obtain	
\begin{align*}
	S_2f(x):&= \left(\int_{|h|}^{\infty} \int_{|x-y|<t}\Big|I_2(x, h, y, t)\Big|^2 \frac{d y d t}{t^{n+1}}\right)^{1 / 2}\\
	&\leq C\|\nabla b\|_{\infty}\left(\int_{|h|}^{\infty} \int_{|x-y|<t}\Big|\int_{|x-z|<2^{\frac{1}{\varepsilon}}|h|}\frac{1}{t^n} |x+h-z|| f(z)| d z\Big|^2 \frac{d y d t}{t^{n+1}}\right)^{1 / 2}\\
	&\leq C\|\nabla b\|_{\infty}(2^{\frac{1}{\varepsilon}}+1)|h|\int_{|x-z|<2^{\frac{1}{\varepsilon}}|h|}\left(\int_{|h|}^{\infty} \int_{|x-y|<t}\frac{d y d t}{t^{3n+1}}\right)^{1 / 2}|f(z)|d z\\
	&\leq C\|\nabla b\|_{\infty}(2^{\frac{1}{\varepsilon}}+1)|h|\frac{2^{\frac{n}{\varepsilon}}}{\big(2^{\frac{1}{\varepsilon}}|h|\big)^n}\int_{|x-z|<2^{\frac{1}{\varepsilon}}|h|}|f(z)|d z\\
	&\leq C(2^{\frac{1}{\varepsilon}}+1)2^{\frac{n}{\varepsilon}}|h| Mf(x).
\end{align*}
Finally, we deal with the term $I_3$, it's easy to see that
\begin{align*}
	S_3f(x):&= \left(\int_{|h|}^{\infty} \int_{|x-y|<t}\Big|I_3(x, h, y, t)\Big|^2 \frac{d y d t}{t^{n+1}}\right)^{1 / 2}\\
	&\leq\|b\|_{\infty} \int_{|x-z|\geq2^{\frac{1}{\varepsilon}}|h|}\left(\int_{|h|}^{\infty} \int_{|x-y|<t} \left|q_{t,\gamma}(y+h,z)-q_{t,\gamma}(y,z)\right|^2  \frac{d y d t}{t^{n+1}}\right)^{1 / 2}|f(z)|d z\\
	&\leq C \int_{|x-z|\geq2^{\frac{1}{\varepsilon}}|h|}\left(\int_F\int_{|x-y|}^{\infty} \left|q_{t,\gamma}(y+h,z)-q_{t,\gamma}(y,z)\right|^2  \frac{d td y }{t^{n+1}}\right)^{1 / 2}|f(z)|d z\\
	&\quad+ C \int_{|x-z|\geq2^{\frac{1}{\varepsilon}}|h|}\left(\int_{F^c}\int_{|x-y|}^{\infty} \left|q_{t,\gamma}(y+h,z)-q_{t,\gamma}(y,z)\right|^2  \frac{d td y }{t^{n+1}}\right)^{1 / 2}|f(z)|d z,
\end{align*}	
where $F:=\{y\in\rn: |x-y|\geq\frac{|x-z|}{2}\}$. Making use of \eqref{3.5} and the fact that $|x-z|\leq t+|y-z|$, it holds
%then for any $y\in F$, $t\geq \frac{|x-z|}{2}$, therefore 

\begin{align}\label{3.24}
	&\int_F\int_{|x-y|}^{\infty} \left|q_{t,\gamma}(y+h,z)-q_{t,\gamma}(y,z)\right|^2  \frac{d td y }{t^{n+1}}\nonumber\\
	&\leq C \int_{|x-y|\geq\frac{|x-z|}{2}}\int_{|x-y|}^{\infty} \frac{1}{t^{2n}}\left(\frac{|h|}{|x-z|}\right)^{2\alpha}  \frac{d td y }{t^{n+1}}\nonumber\\
	&\leq C \frac{|h|^{2\alpha}}{|x-z|^{2\alpha}} \int_{|x-y|\geq\frac{|x-z|}{2}} \frac{1}{|x-y|^{3n}} {d y }\nonumber\\
	&\leq C \frac{|h|^{2\alpha}}{|x-z|^{2n+2\alpha}} .
\end{align}
On the other hand,

\begin{equation}\label{3}
	\begin{split}
		& \int_{F^c}\int_{|x-y|}^{\infty} \left|q_{t,\gamma}(y+h,z)-q_{t,\gamma}(y,z)\right|^2  \frac{d td y }{t^{n+1}}\\
		&= \int_{F^c}\int_{|x-y|}^{\frac{|x-z|}{2}} \left|q_{t,\gamma}(y+h,z)-q_{t,\gamma}(y,z)\right|^2  \frac{d td y }{t^{n+1}}\\
		&\quad+\int_{F^c}\int_{\frac{|x-z|}{2}}^{\infty} \left|q_{t,\gamma}(y+h,z)-q_{t,\gamma}(y,z)\right|^2  \frac{d td y }{t^{n+1}}.
	\end{split}
\end{equation}
Noting that for any $y\in F^c$, $|y-z|\leq \frac{|x-z|}{2}$ and using \eqref{3.5} and the estimate $e^{-s}\leq\frac{C}{s^{N/2}}$ with $2n<N<3n$,  it yields
\begin{align}\label{3.25}
	&\int_{F^c}\int_{|x-y|}^{\frac{|x-z|}{2}} \left|q_{t,\gamma}(y+h,z)-q_{t,\gamma}(y,z)\right|^2  \frac{d td y }{t^{n+1}}\\
	&\leq C \int_{|x-y|\leq\frac{|x-z|}{2}}\int_{|x-y|}^{\infty} \frac{|h|^{2\alpha}}{|x-z|^{2\alpha}}\frac{t^N}{|x-z|^N}  \frac{d td y }{t^{3n+1}}\nonumber\\
	&\leq C \frac{|h|^{2\alpha}}{|x-z|^{2\alpha}}\frac{1}{|x-z|^N} \int_{|x-y|\leq\frac{|x-z|}{2}}  \frac{d y }{|x-y|^{3n-N}}\nonumber\\
	&\leq C \frac{|h|^{2\alpha}}{|x-z|^{2\alpha}}\frac{1}{|x-z|^N} |x-z|^{N-2n}\nonumber\\
	&= C \frac{|h|^{2\alpha}}{|x-z|^{2n+2\alpha}}, \nonumber
\end{align}
and
\begin{align}\label{3.26}
	&\int_{F^c}\int_{\frac{|x-z|}{2}}^{\infty} \left|q_{t,\gamma}(y+h,z)-q_{t,\gamma}(y,z)\right|^2  \frac{d td y }{t^{n+1}}\nonumber\\
	&\leq C\int_{F^c}\int_{\frac{|x-z|}{2}}^{\infty}\frac{|h|^{2\alpha}}{|x-z|^{2\alpha}} \frac{d td y }{t^{3n+1}}\\
	&\leq C\frac{|h|^{2\alpha}}{|x-z|^{2\alpha}}\int_{|x-y|\leq\frac{|x-z|}{2}} \frac{1 }{|x-z|^{3n}}d y\nonumber\\
	&\leq C\frac{|h|^{2\alpha}}{|x-z|^{2n+2\alpha}}\nonumber.
\end{align}
This, combining with \eqref{3.24}, \eqref{3} and \eqref{3.25}, yields that 
\begin{align*}
	S_3	&\leq C \int_{|x-z|\geq2^{\frac{1}{\varepsilon}}|h|}\frac{|h|^{\alpha}}{|x-z|^{n+\alpha}}|f(z)|d z\\
	&=C\sum_{j=0}^{\infty} \int_{2^{j+\frac{1}{\varepsilon}}|h|<|x-z|<2^{j+\frac{1}{\varepsilon}+1}|h|}\frac{|h|^{\alpha}}{|x-z|^{n+\alpha}}|f(z)|d z\\
	&=C|h|^{\alpha}\sum_{j=0}^{\infty} \frac{1}{(2^{j+\frac{1}{\varepsilon}}|h|)^{n+\alpha}}\int_{|x-z|<2^{j+\frac{1}{\varepsilon}+1}|h|}|f(z)|d z\\
	&\leq C|h|^{\alpha}\sum_{j=0}^{\infty} \frac{1}{(2^{j+\frac{1}{\varepsilon}}|h|)^{\alpha}}Mf(x)\\
	&\leq C {2^{-\frac{\alpha}{\varepsilon}}}Mf(x).
\end{align*}	
By the $L^{p, \lambda}(w)$-boundedness of $S_H$ and $M$, we have
$$
\begin{aligned}
	&\left\|\left[b, S_{H,\gamma}\right] f(\cdot+h)-\left[b, S_{H,\gamma}\right] f(\cdot)\right\|_{L^{p, \lambda}(w)}\\
	&\leq \|(S_1+S_2+S_3)f\|_{L^{p, \lambda}(w)}\\
	& \leqslant C |h|\|S_Hf\|_{L^{p, \lambda}(w)}+\big((2^{\frac{1}{\varepsilon}}+1)2^{\frac{n}{\varepsilon}}|h|+2^{-\frac{\alpha}{\varepsilon}}\big)\|Mf\|_{L^{p, \lambda}(w)} \\
	& \leqslant C \big(|h|+(2^{\frac{1}{\varepsilon}}+1)2^{\frac{n}{\varepsilon}}|h|+2^{-\frac{\alpha}{\varepsilon}}\big)\|f\|_{L^{p, \lambda}(w)} .
\end{aligned}
$$
Taking $\varepsilon=\frac{4n}{\log\frac{1}{|h|}} $
and letting $|h|\to 0$,  we have the uniform equicontinuity (b) of Lemma \ref{compactcondition} and complete the proof of Theorem \ref{compact thm} for $S_{H,b}$. 

Finally, to prove Theorem \ref{compact thm} for $S_{P,b}$,  it suffices to note that 
$$
e^{-t\sqrt{L}}=\frac{1}{\sqrt{\pi}}\int_0^\infty\frac{e^{-u}}{\sqrt{u}}e^{-\frac{t^2L}{4u}},
$$
by the subordination formula. This, combining with \eqref{smooth2} and \eqref{size2}, produces that
the kernel of $t\sqrt{L}e^{-t\sqrt{L}}$, $\tilde{q}_t$, enjoys the similar properties as $q_t$:
for any $\tilde{\alpha} \in(0, \nu)$, there exists a positive constant $\tilde{C}_1$ such that for all $t \in(0, \infty)$ and almost every $x, y, h \in \mathbb{R}^n$ with $2|h| \leq t+|x-y|$,
\begin{equation}\label{smooth3}
	\begin{aligned}
		& \left|\tilde{q}_{t}(x+h, y)-\tilde{q}_{t}(x, y)\right|+\left|\tilde{q}_{t}(x, y+h)-\tilde{q}_{t}(x, y)\right| \\
		& \quad \leq \frac{\tilde{C}}{t^n}\left(\frac{|h|}{|x-y|}\right)^{\tilde{\alpha}} \left(1+\frac{\tilde{c}|x-y|^{2}}{t^{2}}\right)^{-\frac{n+1}{2}},
	\end{aligned}
\end{equation}
\begin{equation}\label{size3}
	\left|\tilde{q}_t(x, y)\right|+\left|\tilde{q}_t(y, x)\right| \leq \frac{\tilde{C_1}}{t^{n}}\left(1+\frac{\tilde{c}_1|x-y|^{2}}{t^{2}}\right)^{-\frac{n+1}{2}}.
\end{equation} 
Therefore,  one can obtain the compactness for $S_{P,b}$ by repeating the process of $S_{H,b}$ above and this completes the proof of Theorem \ref{compact thm}.

\qed
\begin{remark}
	For $f \in \mathcal{S}\left(\mathbb{R}^n\right)$, we define the Littlewood-Paley square functions $g_P$ and $g_H$ by
	\begin{align*}
		& g_p(f)(x)=\left(\int_0^{\infty}\left|t \sqrt{L} e^{-t \sqrt{L}} f(x)\right|^2 \frac{d t}{t}\right)^{1 / 2}, \\
		& g_h(f)(x)=\left(\int_0^{\infty}\left|t^2 L e^{-t^2 L} f(x)\right|^2 \frac{d t}{t}\right)^{1 / 2} .
	\end{align*}
	Then, their commutators have the analogous statements as in  Theorem \ref{lebesgue bound}, Theorem \ref{morreybound}, Theorem \ref{compact thm}   and Theorem \ref{wgbF-S} replacing $S_P$, $S_H$ by $g_P, g_H$, respectively.
\end{remark}

%\iffalse
\bigskip \bigskip

C. Zhang, 
Department of Mathematics,
Zhejiang University of Science and Technology,
Hangzhou 310023, People's Republic of China

\smallskip

\noindent{{\it E-mail:}}
\texttt{cmzhang@zust.edu.cnn}

\bigskip

X. Tao,
 Department of Mathematics,
Zhejiang University of Science and Technology,
Hangzhou 310023, People's Republic of China

\smallskip

\noindent{\it E-mail:} \texttt{xxtao@zust.edu.cn}
%\fi


\begin{thebibliography}{99}
		\small
\bibitem{Aus}
Auscher P.,  Duong X.T.,  McIntosh A.: {\it  Boundedness of Banach space valued singular integral operators and Hardy spaces},
Unpublished preprint (2005).

\vspace{-0.2cm}
\bibitem{AHLMT}
Auscher P., Hofmann S., Lacey M., McIntosh A., Tchamitchian, P.: \textit{The
	solution of the Kato square root problem for second order elliptic operators on $\mathbb R^n$}. Ann. of Math. (2) 156 (2002), no. 2, 633-654. %DOI 10.2307/3597201.

\vspace{-0.2cm}
\bibitem{AT}
Auscher P., Tchamitchian Ph.: \textit{Square root problem for divergence operators and related topics}. Ast\'erisque \textbf{249}, viii+172 pp (1998). %http://www.numdam.org/item/AST\_1998\_249\_R1\_0/

\vspace{-0.2cm}
\bibitem{BDMT2} B\'enyi A.,  Dami\'an W., Moen K.,  Torres R. H.:
{\itshape Compact bilinear commutators: The weighted case}. Michigan Math. J. {\bf64} (1), 39-51 (2015). %https://doi.org/10.1307/mmj/1427203284

\vspace{-0.2cm}
\bibitem{BHS}   Bongioanni B.,  Harboure E.,  Salinas O.: {\itshape Commutators of Riesz transforms related to Schr\"odinger operators}. J. Fourier Anal. Appl. {\bf17} (1), 115-134 (2011). %https://doi.org/10.1007/s00041-010-9133-6

\vspace{-0.2cm}
\bibitem{BCD}
Bui T.,  Cao J., Ky L.,  Yang D.,  Yang S.: {\it Musielak-Orlicz-Hardy spaces associated with operators satisfying reinforced off-diagonal estimates}. Anal. Geom. Metr. Spaces \textbf{1}, 69-129 (2013). %https://doi.org/10.1007/s12220-012-9344-y

\vspace{-0.2cm}
\bibitem{Ca}
Calder\'on A. P.: \textit{Commutators of singular integral operators}, Proc. Nat. Acad. Sci. U.S.A. \textbf{53}
(1965), 1092-1099.

\vspace{-0.2cm}
\bibitem{COY1}
Cao M., Olivo A., Yabuta K.: {\itshape Extrapolation for multilinear compact operators and applications}. Trans. Amer. Math. Soc. {\bf375} (7), 5011-5070 (2022).% https://doi.org/10.1090/tran/8645

\vspace{-0.2cm}
\bibitem{CSZ}
Cao M., Si Z., Zhang J.: \textit{Weak and strong type estimates for square functions associated with operators}. J. Math. Anal. Appl. \textbf{527}, no. 1, part 1, Paper No. 127369, 29pp (2023). %https://doi.org/10.1016/j.jmaa.2023.127369

\vspace{-0.2cm}
\bibitem{Ch} 
Chen Y.,  Ding Y.:, \textit{$L^2$ boundedness for commutator of rough singular integral
	with variable kernel}, Rev. Mat. Iberoam. \textbf{24}, no.2, 531-547 (2008).% DOI 10.4171/RMI/545.

\vspace{-0.2cm}
\bibitem{CD1} 
Chen Y.,  Ding Y.:
{\itshape Compactness of commutators of singular integrals with variable kernels}. (Chinese)
Chinese Ann. Math. Ser. A {\bf30}, 201-212  (2009); translation in Chinese J. Contemp. Math. {\bf30} (2), 153-166 (2009).  

\vspace{-0.2cm}

\bibitem{CDW1}  Chen Y.,  Ding Y.,  Wang X.:
{\itshape Compactness for commutators of Marcinkiewicz integral in Morrey spaces}.
Taiwanese Jour. Math. {\bf15}, 633-658 (2011).  %http://www.tjm.nsysu.edu.tw/


\vspace{-0.2cm}
\bibitem{CW15}
Chen Y., Wang H.: \textit{ Compactness for the commutator of the parameterized area integral in the Morrey space}. Math. Inequal. Appl. \textbf{18} (4),  1261-1273 (2015). %{https://api.semanticscholar.org/CorpusID:59019947}

\vspace{-0.2cm}

\bibitem{CF87}  Chiarenza F.,  Frasca M.: {\it Morrey spaces and Hardy-Littlewood maximal function}. Rend. Mat. Appl. {\bf7}, 273-279 (1987).


\vspace{-0.2cm}
\bibitem{fm}
Chiarenza F.,  Frasca M., Longo P.: \textit{$W^{2,p}$-solvability of the Dirichlet problem for nondivergence elliptic equations with VMO coefficients}, Trans. Amer. Math. Soc. \textbf{336} (2), 841-853 (1993). %DOI 10.2307/2154379.

\vspace{-0.2cm}
\bibitem{CPP12}  Chung D., Pereyra M. C., Perez C.:
{\itshape Sharp bounds for general commutators on weighted Lebesgue spaces}.
Trans. Amer. Math. Soc. {\bf364} (3), 1163-1177 (2012). %https://doi.org/10.1090/S0002-9947-2011-05534-0

\vspace{-0.2cm}
\bibitem{CRW}
Coifman R., Rochberg R., Weiss G.: {\it Factorization theorems for Hardy spaces in several variables}. Ann. of Math. {\bf 103}, 611-635 (1976). %https://doi.org/10.2307/1970954

\vspace{-0.2cm}
\bibitem{DM15} Ding Y., Mei T.:
{\itshape Boundedness and compactness for the commutators of bilinear operators on Morrey spaces}.
Potential Anal. {\bf42}, 717-748 (2015). %https://doi.org/10.1007/s11118-014-9455-0 


\vspace{-0.2cm}
\bibitem{fr}
Difazio G., Ragusa M. A.: \textit{Interior estimates in Morrey spaces for strong solutions to
	nondivergence form equations with discontinuous coefficients}, J. Funct. Anal. \textbf{112} (2), 241-256 (1993). %https://doi.org/10.1006/jfan.1993.1032.

\vspace{-0.2cm}
\bibitem{DL10} Duong X. T., Li J.: {\it Hardy spaces associated to operators satisfying Davies-Gaffney estimates and bounded holomorphic functional calculus}. J. Funct. Anal. {\bf 264} (6), 1409-1437 (2013). %https://doi.org/10.1016/j.jfa.2013.01.006


\vspace{-0.2cm}
\bibitem{TY} Duong X. T., Yan L.:
{\it Duality of Hardy and BMO spaces associated with operators with heat kernel bounds}.
J. Amer. Math. Soc. {\bf 18} (4), 943-973 (2005).  %https://doi.org/10.1090/S0894-0347-05-00496-0 


\vspace{-0.2cm}
\bibitem{fs24} Fefferman C., Stein E. M.: {\it  $H^p$
	spaces of several variables}. Acta Math. {\bf 129},  137-193 (1972).% https://doi.org/10.1007/BF02392215

\vspace{-0.2cm}
\bibitem{GX} Gong R., Xie P.:
{\it Weighted $L^p$ estimates for the area integral associated with self-adjoint operators on homogeneous space}.
J. Math. Anal. Appl.
{\bf 393} (2),
590-604 (2012). %https://doi.org/10.1016/j.jmaa.2012.04.038

\vspace{-0.2cm}

\bibitem{Gong}
Gong R.: 
{\it The area integral associated to self-adjoint operators on weighted Morrey spaces}.
J. Math. Sci. Adv. Appl. {\bf16} (1-2), 47-59 (2012). %https://api.semanticscholar.org/CorpusID:54952707

\vspace{-0.2cm}

\bibitem{GY}
Gong R., Yan L., 
{\it Weighted $L^p$ estimates for the area integral associated to self-adjoint operators}. 
Manuscripta Math. {\bf144} (1-2), 25-49 (2014). %https://doi.org/10.1007/s00229-013-0639-5  

\vspace{-0.2cm}

\bibitem{HLM} Hofmann S., Lu G., Mitrea D., Mitrea M., Yan L.: {\it Hardy spaces associated to non-negative self-adjoint operators satisfying Davies-Gaffney estimates}. Mem. Amer. Math. Soc. {\bf214} (1007), vi+78 pp (2011). %https://doi.org/10.1090/S0065-9266-2011-00624-6


\vspace{-0.2cm}
\bibitem{HM} Hofmann S., Mayboroda S.: {\it Hardy and BMO spaces associated to divergence form elliptic operators}. Math. Ann. {\bf344} (1), 37-116 (2009).% https://doi.org/10.1007/s00208-008-0295-3

\vspace{-0.2cm}

\bibitem{H1} Hyt\"onen T., Stefanos L.:	{\itshape Extrapolation of compactness on weighted spaces}. Rev. Mat. Iberoam. {\bf39} (1), 91-122 (2023).% https://doi.org/10.4171/RMI/1325

\vspace{-0.2cm}

\bibitem{H3} Hyt\"onen T., Lappas S.: 	{\itshape Extrapolation of compactness on weighted spaces: bilinear operators}.
Indag. Math. {\bf 33} (2),  397-420 (2022). %https://doi.org/10.1016/j.indag.2021.09.007

\vspace{-0.2cm}

\bibitem{KL2} Krantz S., Li S. Y.:
{\itshape Boundedness and compactness of integral operators on spaces of homogeneous
	type and applications II}.
J. Math. Anal. Appl. {\bf258}, 642-657 (2001). %https://doi.org/10.1006/jmaa.2000.7403 

\vspace{-0.2cm}
\bibitem{LOR17}
Lerner A. K.,  Ombrosi S., Rivera-R\'ios  I. P.:
\emph{On pointwise and weighted estimates for commutators of Calder\'on-Zygmund operators.}
Adv. Math. \textbf{319}, 153-181 (2017). %https://doi.org/10.1016/j.aim.2017.08.022 

\vspace{-0.2cm}
\bibitem{LOR21}
Lerner A. K. , Ombrosi  S.,  Rivera-R\'ios I. P.: \textit{On two weight estimates for iterated commutators}. J. Funct. Anal. \textbf{281} (8), Paper No. 109153, 46 pp (2021). %https://doi.org/10.1016/j.jfa.2021.109153

\vspace{-0.2cm}
\bibitem{LC}
Liu F., Cui P.: \textit{Variation operators for singular integrals and their commutators on
	weighted Morrey spaces and Sobolev spaces}. Sci. China Math. \textbf{65}, 1267-1292 (2022). %https://doi.org/10.1007/s11425-020-1828-6 

\vspace{-0.2cm}
\bibitem{MP1}
Martell J. M., Prisuelos-Arribas C.: \textit{Weighted Hardy spaces associated with elliptic operators. Part I: weighted norm
	inequalities for conical square functions}. Trans. Am. Math. Soc. \textbf{369}, 4193-4233 (2017). %http://dx.doi.org/10.1090/tran/6768

\vspace{-0.2cm}
\bibitem{MP2} Martell J. M., Prisuelos-Arribas C.: \textit{Weighted Hardy spaces associated with elliptic operators. Part II: characterizations of $H_L^1(w)$}. Publ. Mat. \textbf{62}, 475-535 (2018). %https://doi.org/10.5565/PUBLMAT6221806


\vspace{-0.2cm}
\bibitem{O}
O'Neil, R.:\textit{ Fractional integration in Orlicz spaces I.} Trans. Amer. Math. Soc. \textbf{115},
300-328 (1965).


\vspace{-0.2cm}
\bibitem{P95}
P\'erez C.: {\it Endpoint estimates for commutators of singular integral operators}. J. Funct. Anal. \textbf{128}, 163-185 (1995). %https://doi.org/10.1006/jfan.1995.1027

\vspace{-0.2cm}
\bibitem{P13}
P\'erez C.: {\it A course on singular integrals and weights}. Adv. Courses Math. CRM Barcelona, Birkh\"auser, Basel (2013). %https://doi.org/10.1007/978-3-0348-0408-0\_3

\vspace{-0.2cm}
\bibitem{[PT]}
{ P\'{e}rez C., Trujillo-Gonz\'{a}lez R.:} {\it Sharp weighted estimates for multilinear commutators}.
{J. London Math. Soc.(2)} {\bf65} (3), {672-692} (2002). %https://doi.org/10.1112/S0024610702003174

\vspace{-0.2cm}

\bibitem{RR}
Rao, M. M., Ren, Z. D.: {\it Theory of Orlicz spaces}. Monographs and Textbooks in Pure
and Applied Mathematics, \textbf{146}. Marcel Dekker, Inc., New York, 1991.

\vspace{-0.2cm}

\bibitem{SY}
Song L., Yan L.: {\it A maximal function characterization for Hardy spaces associated to nonnegative self-adjoint operators satisfying Gaussian estimates}. Adv. Math. {\bf287}, 463-484 (2016). %https://doi.org/10.1016/j.aim.2015.09.026



\vspace{-0.2cm}

\bibitem{SW58}
Stein E.M., Weiss G.: {\it Interpolation of operators with change of measures}.
Trans. Amer. Math. Soc. {\bf87} 159-172 (1958). %https://doi.org/10.2307/1993094	

\vspace{-0.2cm}

\bibitem{TXYY}
Tao J., Xue Q., Yang D., Yuan W.:
{\itshape XMO and weighted compact bilinear commutators}. J. Fourier Anal. Appl. {\bf27} (3), Paper No. 60, 34 pp (2021). %https://doi.org/10.1007/s00041-021-09854-x

\vspace{-0.2cm}

\bibitem{TX}
Torres R. H., Xue Q.:
{\itshape On compactness of commutators of multiplication and bilinear pesudodifferential operators and a new subspace of $BMO$}. Rev. Mat. Iberoam. {\bf36} (3), 939-956 (2020).% https://doi.org/10.4171/rmi/1156

\vspace{-0.2cm}
\bibitem{Uch} Uchiyama A.:
{\itshape On the compactness of operators of Hankel type}. Tohoku Math. J.(2) {\bf 30} (1), 163-171 (1978). %https://doi.org/10.2748/tmj/1178230105

\vspace{-0.2cm}
\bibitem{uhl}
Uhl M., {\it Spectral multiplier theorems of H\"ormander type via generalized Gaussian
	estimates,} Ph.D. Dissertation, Karlsruher Institut f\"ur Technologie (KIT), Karlsruhe
(2011). %URL: http://digbib.ubka.uni-karlsruhe.de/volltexte/1000025107


\vspace{-0.2cm}

\bibitem{DX} Xue  Q., Ding Y.: \textit{Weighted estimates for the multilinear commutators of the Littlewood-Paley operators}. Sci. China Ser. A-Math. \textbf{52}, 1849-1868 (2009). %https://doi.org/10.1007/s11425-009-0049-z

\vspace{-0.2cm}

\bibitem{XZ}Xue, Q., Zhang, C.: \textit{On weighted compactness of commutators of Stein's square functions associated with
	Bochner-Riesz means}. J. Geom. Anal. \textbf{34} (2024), no. 11, Paper No. 332, 18 pp. %https://doi.org/10.1007/s12220-024-01775-7

\end{thebibliography}
\end{document}